\RequirePackage{fix-cm}
\documentclass{siamart171218}
\newsiamremark{remark}{Remark}
\usepackage{algorithmic} 
\Crefname{ALC@unique}{Line}{Lines} 

\usepackage{amsmath}
\usepackage{amssymb}
\usepackage{mathtools}
\usepackage{cases}
\usepackage{fancyvrb}
\VerbatimFootnotes
\usepackage{pgfplots, pgfplotstable, booktabs, colortbl, array}
\usepackage{bibentry}
\nobibliography*
\usepackage{enumitem}

\usepackage[T1]{fontenc}

\usepackage{pgfplots, pgfplotstable, booktabs, colortbl, array}
\usepackage{tikz}
\pgfplotsset{compat=newest}
\usepackage{pgffor}
\usepackage{xcolor}

\DeclareMathOperator*{\argmin}{arg\,min}
\newcommand*{\mysqrt}[2][2]{\sqrt[\leftroot{0}\uproot{3}#1]{#2}}

\renewcommand{\Re}{\operatorname{Re}}
\renewcommand{\Im}{\operatorname{Im}}

\usepackage{tikz}
\usetikzlibrary{positioning,arrows}

\begin{document}

\title{Rational Minimax Iterations for Computing the Matrix \MakeLowercase{$p$}th Root}

\author{Evan S. Gawlik\thanks{Department of Mathematics, University of Hawaii at Manoa
    (\email{egawlik@hawaii.edu})}}
    
\date{}

\headers{Rational minimax iterations for computing the matrix \MakeLowercase{$p$}th root}{E. S. Gawlik}

\maketitle

\begin{abstract}
In~[\bibentry{gawlik2018zolotarev}], a family of iterations for computing the matrix square root was constructed by exploiting a recursion obeyed by Zolotarev's rational minimax approximants of the function $z^{1/2}$.  The present paper generalizes this construction by deriving rational minimax iterations for the matrix $p^{th}$ root, where $p \ge 2$ is an integer.  The analysis of these iterations is considerably different from the case $p=2$, owing to the fact that when $p>2$, rational minimax approximants of the function $z^{1/p}$ do not obey a recursion.  Nevertheless, we show that several of the salient features of the Zolotarev iterations for the matrix square root, including equioscillatory error, order of convergence, and stability, carry over to case $p>2$.  A key role in the analysis is played by the asymptotic behavior of rational minimax approximants on short intervals.  Numerical examples are presented to illustrate the predictions of the theory.
\end{abstract}


\section{Introduction} \label{sec:intro}

In recent years, a growing body of literature has highlighted the usefulness of \emph{rational minimax iterations} for computing functions of matrices~\cite{nakatsukasa2010optimizing,nakatsukasa2016computing,gawlik2018zolotarev,gawlik2018backward,byers2008new}.   In these studies, $f(A)$ is approximated by a rational function $r$ of $A$ possessing two properties: $r$ closely (and often optimally) approximates $f$ in the uniform norm over a subset of the real line, and $r$ can be generated from a recursion.  A prominent example of such an iteration was introduced by Nakatsukasa and Freund in~\cite{nakatsukasa2016computing}, where it was observed that rational minimax approximants of the function $\mathrm{sign}(z) = z/(z^2)^{1/2}$ obey a recursion, allowing one to rapidly compute $\mathrm{sign}(A)$ and related decompositions such as the polar decomposition, symmetric eigendecomposition, SVD, and, in subsequent work, the CS decomposition~\cite{gawlik2018backward}.  An analogous recursion for rational minimax approximants of $z^{1/2}$ has recently been used to construct iterations for the matrix square root~\cite{gawlik2018zolotarev}, building upon ideas of Beckermann~\cite{beckermann2013optimally}.  There, the iterations are referred to as \emph{Zolotarev iterations}, owing to the role played by explicit formulas for rational minimax approximants of $\mathrm{sign}(z)$ and $z^{1/2}$ derived by Zolotarev~\cite{zolotarev1877applications}.

The aim of this paper is to introduce a family of rational minimax iterations for computing the principal $p^{th}$ root $A^{1/p}$ of a square matrix $A$, where $p \ge 2$ is an integer.  
Recall that the principal $p^{th}$ root of a square matrix $A$ having no nonpositive real eigenvalues is the unique solution of $X^p=A$ whose eigenvalues are contained in $\{z \in \mathbb{C} \mid -\pi/p < \arg z < \pi/p\}$~\cite[Theorem 7.2]{higham2008functions}.
The iterations we propose reduce to the Zolotarev iterations for the matrix square root~\cite{gawlik2018zolotarev} when $p=2$, but when $p>2$, they differ from the Zolotarev iterations in several important ways.    
Notably, for all integers $p \ge 2$, the iterations generate a rational function of $r$ of $A$ which has the property that for scalar inputs, the relative error $e(z)=(r(z)-z^{1/p})/z^{1/p}$ equioscillates on a certain interval $[a,b]$ (see Section~\ref{sec:background} for our terminology).  Remarkably, when $p=2$, $e(z)$ equioscillates often enough to render $\max_{a \le z \le b} |e(z)|$ minimal among all choices of $r$ with a fixed numerator and denominator degree~\cite{gawlik2018zolotarev}.  This optimality property is the hallmark of the Zolotarev iterations, and it allows one to appeal to classical results from rational approximation theory to estimate the maximum relative error.  When $p>2$, no such optimality property holds.  
Much of this paper is devoted to showing that the rational minimax iterations for the $p^{th}$ root still enjoy many of the same desirable features as the Zolotarev iterations for the square root, despite the absence of optimality in the case $p>2$.  We take care to present our results in such a way that when $p=2$, the salient features of the Zolotarev iterations are recovered as special cases.

There are a number of connections between the iterations we derive and existing iterations from the literature on the matrix $p^{th}$ root. 
We have already mentioned that they reduce to the Zolotarev iterations when $p=2$.  For arbitrary $p \ge 2$, the two lowest order versions of our rational minimax iterations are scaled variants of the Newton iteration and the inverse Newton iteration~\cite[Chapter 6]{higham2008functions},~\cite[Section 6]{bini2005algorithms},~\cite{iannazzo2006newton}.  In another limiting case, our iterations reduce to the Pad\'e iterations~\cite[Section 5]{laszkiewicz2009pade}.  
Relative to these iterations, the rational minimax iterations offer advantages primarily when the matrix $A$ has eigenvalues with widely varying magnitudes.  As an extreme example, if $p=3$ and $A$ is Hermitian positive definite with condition number $\le 10^{16}$, convergence is achieved in double-precision arithmetic after just $2$ iterations when using our type-$(6,6)$ rational minimax iteration.  In contrast, up to $5$ iterations are needed when using the type-$(6,6)$ Pad\'e iteration.  Our numerical experiments indicate that the situation is similar, but less dramatic, for non-normal matrices with eigenvalues away from the positive real axis.

This paper is organized as follows.  In Section~\ref{sec:background}, we review the Zolotarev iterations for the matrix square root by summarizing the contents of~\cite{gawlik2018zolotarev}.  In Section~\ref{sec:pthroot}, we introduce rational minimax iterations for the matrix $p^{th}$ root and present our main results: Theorem~\ref{thm:fit}, Theorem~\ref{thm:pthroot}, and their corollaries.  Proofs of these results are provided separately in Section~\ref{sec:proofs}.  Finally, Section~\ref{sec:numerical} presents numerical experiments that illustrate the predictions of the theory.

\section{Background: Zolotarev iterations for the matrix square root} \label{sec:background}

Let us summarize the Zolotarev iterations for the matrix square root and their key properties~\cite{gawlik2018zolotarev}.  Let $\mathcal{R}_{m,\ell}$ denote the set of all rational functions of type $(m,\ell)$ -- ratios of polynomials of degree $\le m$ to polynomials of degree $\le \ell$.  We say that a function $r(z) = g(z)/h(z)$ in $\mathcal{R}_{m,\ell}$ has \emph{exact type} $(m',\ell')$ if, after canceling common factors, $g(z)$ and $h(z)$ have degree exactly $m' \le m$ and $\ell' \le \ell$, respectively.  The number $d = \min\{m-m',\ell-\ell'\}$ is called the \emph{defect} of $r$ in $\mathcal{R}_{m,\ell}$.  In most of what follows, $z$ is a real variable; we use the letter $z$ since the behavior of $r$ on $\mathbb{C}$ will play an important role later in the paper.

Given a continuous, increasing bijection $f : [0,1] \rightarrow [0,1]$ and a number $\alpha \in (0,1)$, let  $r_{m,\ell}(z, \alpha, f)$ denote the best type-$(m,\ell)$ rational approximant of $f(z)$ on $[f^{-1}(\alpha),1]$:
\begin{equation} \label{rml}
r_{m,\ell}(\cdot, \alpha, f) = \argmin_{r \in \mathcal{R}_{m,\ell}} \max_{z \in [f^{-1}(\alpha),1]} \left| \frac{r(z)-f(z)}{f(z)} \right|.
\end{equation}
It is well-known that the minimization problem above has a unique solution~\cite[p. 55]{akhiezer1956theory}.
Furthermore, explicit formulas for $r_{m,\ell}(\cdot, \alpha, \sqrt{\cdot})$ are known for $\ell \in \{m-1,m\}$~\cite{zolotarev1877applications}.
Let $\hat{r}_{m,\ell}(z, \alpha, f)$ denote the unique scalar multiple of $r_{m,\ell}(z,\alpha,f)$ with the property that
\begin{equation} \label{rhat}
\min_{z \in [f^{-1}(\alpha),1]} \frac{ \hat{r}_{m,\ell}(z,\alpha,f) - f(z) }{f(z)} = 0.
\end{equation}
For $m \in \mathbb{N}$ and $\ell \in \{m-1,m\}$, the Zolotarev iteration of type $(m,\ell)$ for computing the square root of a square matrix $A$ reads
\begin{align} 
X_{k+1} &= X_k \hat{r}_{m,\ell}\left( X_k^{-2}A, \alpha_k, \sqrt{\cdot} \right), & X_0 &= I, \label{sqrtmit1} \\
\alpha_{k+1} &= \frac{\alpha_k}{\hat{r}_{m,\ell}(\alpha_k^2,\alpha_k,\sqrt{\cdot})}, & \alpha_0 &= \alpha. \label{sqrtmit2}
\end{align}
It is proven in~\cite{gawlik2018zolotarev} that in exact arithmetic, $X_k \rightarrow A^{1/2}$ and $\alpha_k \rightarrow 1$ with order of convergence $m+\ell+1$ for any $A$ with no nonpositive real eigenvalues.  In floating point arithmetic, it is necessary to reformulate the iteration to ensure its stability; we detail the stable reformulation of~(\ref{sqrtmit1}-\ref{sqrtmit2}) later on.

The iteration~(\ref{sqrtmit1}-\ref{sqrtmit2}) has the remarkable property that it generates an optimal rational approximation of $A^{1/2}$ of high degree. Namely, $\widetilde{X}_k := 2\alpha_k X_k/(1+\alpha_k) = r_{m_k,\ell_k}(A,\alpha,\sqrt{\cdot})$, where
\begin{equation} \label{mklk}
(m_k,\ell_k) = \begin{cases}
\left(\frac{1}{2}(2m)^k,\frac{1}{2}(2m)^k-1\right), &\mbox{ if } \ell=m-1, \\
\left(\frac{1}{2}((2m+1)^k-1),\frac{1}{2}((2m+1)^k-1)\right), &\mbox{ if } \ell=m. \\
\end{cases}
\end{equation}
A simple consequence of this is that if $A$ is Hermitian positive definite with eigenvalues in $[\alpha^2,1]$, then
\[
\|(\widetilde{X}_k-A^{1/2})A^{-1/2}\|_2 \le E_{m_k,\ell_k}(\sqrt{\cdot},[\alpha^2,1]),
\] 
where
\[
E_{m,\ell}(f,S) = \min_{r \in \mathcal{R}_{m,\ell}} \max_{z \in S} \left| \frac{r(z)-f(z)}{f(z)} \right|.
\]
For more detailed error estimates, including error estimates for non-normal $A$ with eigenvalues in $\mathbb{C} \setminus (-\infty,0]$, see~\cite{gawlik2018zolotarev}.

\section{Minimax iterations for the matrix $p^{th}$ root} \label{sec:pthroot}

In this paper, we propose an iteration for computing $p^{th}$ roots of matrices that generalizes~(\ref{sqrtmit1}-\ref{sqrtmit2}).  Given $\alpha \in (0,1)$, $m,\ell \in \mathbb{N}_0$, and an integer $p \ge 2$, the iteration reads
\begin{align} 
X_{k+1} &= X_k \hat{r}_{m,\ell}\left( X_k^{-p}A, \alpha_k, \mysqrt[p]{\cdot} \right), & X_0 &= I, \label{pthrootmit1} \\
\alpha_{k+1} &= \frac{\alpha_k}{\hat{r}_{m,\ell}(\alpha_k^p,\alpha_k,\mysqrt[p]{\cdot})}, & \alpha_0 &= \alpha. \label{pthrootmit2}
\end{align}
The Zolotarev iterations~(\ref{sqrtmit1}-\ref{sqrtmit2}) correspond to the cases $\{(m,\ell,p) \mid m \in \mathbb{N}, \, \ell \in \{m-1,m\}, \, p=2\}$ in~(\ref{pthrootmit1}-\ref{pthrootmit2}).  (Note that we abusively referred to these cases as ``the case $p=2$'' in Section~\ref{sec:intro}).

With the exception of the cases $\{(m,\ell,p) \mid m \in \mathbb{N}, \, \ell \in \{m-1,m\}, \, p=2\}$ and $\{(m,\ell,p) \mid (m,\ell) \in \{(0,0),(1,0),(0,1)\}, \, p \ge 2\}$,
explicit formulas for $\hat{r}_{m,\ell}(z,\alpha,\mysqrt[p]{\cdot})$ are not known.
However, $\hat{r}_{m,\ell}(z,\alpha,\mysqrt[p]{\cdot})$ can be computed numerically; see Section~\ref{sec:numerical} for details.  Note that the cost of computing $\hat{r}_{m,\ell}(z,\alpha,\mysqrt[p]{\cdot})$ is independent of the dimension of $A$, so it is expected to be negligible for problems involving large matrices. 

As with the square root iteration~(\ref{sqrtmit1}-\ref{sqrtmit2}), it is necessary to reformulate the $p^{th}$ root iteration~(\ref{pthrootmit1}-\ref{pthrootmit2}) to ensure its stability.  This is accomplished by considering the iteration for $Y_k = X_k^{1-p}A$ and $Z_k = X_k^{-1}$ implied by~(\ref{pthrootmit1}-\ref{pthrootmit2}).  Exploiting commutativity, we have
\begin{align} 
Y_{k+1} &= Y_k h_{\ell,m,p}\left( Z_k Y_k, \alpha_k \right)^{p-1}, & Y_0 &= A, \label{pcoupled1} \\
Z_{k+1} &= h_{\ell,m,p}\left( Z_k Y_k, \alpha_k \right) Z_k, & Z_0 &= I, \label{pcoupled2} \\
\alpha_{k+1} &= \alpha_k h_{\ell,m,p}(\alpha_k^p, \alpha_k), & \alpha_0 &= \alpha, \label{pcoupled3}
\end{align}
where $h_{\ell,m,p}(z,\alpha) = r_{m,\ell}(z,\alpha, \mysqrt[p]{\cdot} )^{-1}$.  (We swapped the order of the first two indices to emphasize that $h_{\ell,m,p}(z,\alpha)$ is a rational function of type $(\ell,m)$, not $(m,\ell)$.)

The remainder of this section presents a series of results about the behavior of the iteration~(\ref{pthrootmit1}-\ref{pthrootmit2}) and its counterpart~(\ref{pcoupled1}-\ref{pcoupled3}).  Proofs of these results are given in Section~\ref{sec:proofs}.

\subsection{Functional iteration}
A great deal of information about the behavior of the iteration~(\ref{pthrootmit1}-\ref{pthrootmit2}) (and hence~(\ref{pcoupled1}-\ref{pcoupled3})) can be gleaned from a study of the functional iteration
\begin{align} 
f_{k+1}(z) &= f_k(z) \hat{r}_{m,\ell}\left( \frac{z}{f_k(z)^p}, \alpha_k, \sqrt[p]{\cdot} \right), & f_0(z) &= 1, \label{pthrootit1} \\
\alpha_{k+1} &= \frac{\alpha_k}{\hat{r}_{m,\ell}(\alpha_k^p,\alpha_k,\sqrt[p]{\cdot})}, & \alpha_0 &= \alpha. \label{pthrootit2}
\end{align}
Indeed, we have $X_k = f_k(A)$ in~(\ref{pthrootmit1}-\ref{pthrootmit2}), and $Y_k=f_k(A)^{1-p}A$ and $Z_k = f_k(A)^{-1}$ in~(\ref{pcoupled1}-\ref{pcoupled3}).

The following theorem summarizes the properties of the functional iteration~(\ref{pthrootit1}-\ref{pthrootit2}).  In the interest of generality, it focuses on a slight generalization of~(\ref{pthrootit1}-\ref{pthrootit2}) that reduces to~(\ref{pthrootit1}-\ref{pthrootit2}) when the function $f$ appearing below is $f(z)=z^{1/p}$.  The theorem makes use of the following terminology.  A continuous function $g(z)$ is said to \emph{equioscillate} $m$ times on an interval $[a,b]$ if there exist $m$ points $a \le z_0 < z_1 < \dots < z_{m-1} \le b$ at which
\[
g(z_j) = \sigma (-1)^j \max_{z \in [a,b]} |g(z)|, \quad j=0,1,\dots,m-1.
\]
for some $\sigma \in \{-1,1\}$.  It is well-known that the minimax approximants~(\ref{rml}) are uniquely characterized by the property that $\frac{r_{m,\ell}(z,\alpha,f)-f(z)}{f(z)}$ equioscillates at least $m+\ell+2-d$ times on $[f^{-1}(\alpha),1]$, where $d$ is the defect of $r_{m,\ell}(z,\alpha,f)$ in $\mathcal{R}_{m,\ell}$~\cite[Theorem 24.1]{trefethen2013approximation}.  We will be particularly interested in those functions $f$ for which:
\begin{enumerate}[label=(\ref*{sec:pthroot}.\roman*),ref=\ref*{sec:pthroot}.\roman*]
\item \label{assumption1} For every $\alpha \in (0,1)$ and $m,\ell \in \mathbb{N}_0$, $r_{m,\ell}(z,\alpha,f)$ has exact type $(m,\ell)$.  Furthermore, $\frac{r_{m,\ell}(z,\alpha,f)-f(z)}{f(z)}$ equioscillates exactly $m+\ell+2$ times on $[f^{-1}(\alpha),1]$, achieves its maximum at $z=f^{-1}(\alpha)$, and achieves an extremum at $z=1$.
\end{enumerate}
The function is $f(z) = z^{1/p}$ satisfies this hypothesis; see Lemma~\ref{lemma:details} for a proof.
\begin{theorem} \label{thm:fit} 
Let $f : [0,1] \rightarrow [0,1]$ be a continuous, increasing bijection satisfying~(\ref{assumption1}).
Let $\alpha \in (0,1)$ and $m,\ell \in \mathbb{N}_0$, and define $f_k(z)$ recursively by
\begin{align} 
f_{k+1}(z) &= f_k(z) \hat{r}_{m,\ell}\left( f^{-1}\left( \frac{f(z)}{f_k(z)} \right), \alpha_k, f \right), & f_0(z) &= 1, \label{genit1} \\
\alpha_{k+1} &= \frac{\alpha_k}{\hat{r}_{m,\ell}(f^{-1}(\alpha_k),\alpha_k,f)}, & \alpha_0 &= \alpha. \label{genit2}
\end{align}
Then, with $\widetilde{f}_k(z) = \frac{2\alpha_k}{1+\alpha_k} f_k(z)$ and $\varepsilon_k = \max_{z \in [f^{-1}(\alpha),1]} \left| \frac{\widetilde{f}_k(z) - f(z)}{f(z)} \right|$, we have:
\begin{enumerate}[label=\textup{(\ref*{sec:pthroot}.\roman*)},ref=\ref*{sec:pthroot}.\roman*,resume]
\item \label{claim1} For every $k \ge 0$,
\begin{equation} \label{alphaeps}
\alpha_k = \frac{1-\varepsilon_k}{1+\varepsilon_k}
\end{equation}
and
\begin{equation} \label{err}
\varepsilon_{k+1} =
 E_{m,\ell}(f, [f^{-1}(\alpha_k),1]).
\end{equation}
\item \label{equi} For every $k \ge 0$, the relative error $\frac{\widetilde{f}_k(z) - f(z)}{f(z)}$ equioscillates $(m+\ell+1)^k+1$ times on $[f^{-1}(\alpha),1]$, and it achieves its extrema at the endpoints.
\item \label{convergence} If $f \in C^{m+\ell+1}([\alpha,1])$, $f^{-1}$ is Lipschitz on $[\alpha,1]$, and $(m,\ell) \neq (0,0)$, then $\varepsilon_k \rightarrow 0$ monotonically with order of convergence $m+\ell+1$ as $k \rightarrow \infty$.
\end{enumerate}
\end{theorem}

Let us discuss the meaning of this theorem.  It states that the iteration~(\ref{genit1}-\ref{genit2}) generates a function $\widetilde{f}_k(z) \approx f(z)$ with the following curious property: The maximum relative error in $\widetilde{f}_k(z)$ on the interval $[f^{-1}(\alpha),1]$ is equal to the maximum relative error in the best rational approximant of $f(z)$ \emph{on a much smaller interval} $[f^{-1}(\alpha_{k-1}),1]$.  Indeed, as $k$ increases, the length of $[f^{-1}(\alpha),1]$ remains constant, whereas the length of $[f^{-1}(\alpha_{k-1}),1]=[f^{-1}(\alpha_{k-1}),f^{-1}(1)]$ is $O(1-\alpha_{k-1}) = O(\varepsilon_{k-1})$ by~(\ref{alphaeps}), assuming $f^{-1}$ is Lipschitz near $z=1$.  Since rational functions of type $(m,\ell)$ can approximate smooth functions on intervals of length $O(\varepsilon_{k-1})$ with accuracy $O(\varepsilon_{k-1}^{m+\ell+1})$, we see from~(\ref{err}) that $\varepsilon_k = O(\varepsilon_{k-1}^{m+\ell+1})$, assuming $f$ is smooth enough near $z=1$. That is, $\varepsilon_k \rightarrow 0$ with order of convergence $m+\ell+1$.

For most functions $f$, the iteration~(\ref{genit1}-\ref{genit2}) is not useful, as it (rather circularly) uses $f$ (and $f^{-1}$) to generate an approximation of $f$.  Furthermore, the approximation it generates need not be a rational function of $z$.  The function $f(z)=z^{1/p}$, however, is exceptional, in that the iteration~(\ref{genit1}-\ref{genit2}) -- which reduces to~(\ref{pthrootit1}-\ref{pthrootit2}) for this $f$ -- generates a rational function $f_k(z)$ without requiring the evaluation of any $p^{th}$ roots.

The following theorem specializes Theorem~\ref{thm:fit} to the case $f(z)=z^{1/p}$ and gives precise information about the constants implicit in the convergence result~(\ref{convergence}).  In it, we use the notation $(\beta)_m$ for the rising factorial (the Pochhammer symbol): $(\beta)_m = \beta(\beta+1)(\beta+2)\cdots(\beta+m-1)$.

\begin{theorem} \label{thm:pthroot}
Let $\alpha \in (0,1)$, $m,\ell \in \mathbb{N}_0$, and $p \in \mathbb{N}$ with $p>2$ and $(m,\ell) \neq 0$.  Let $f_k(z)$ and $\alpha_k$ be defined by the iteration~(\ref{pthrootit1}-\ref{pthrootit2}), and let $\widetilde{f}_k(z) = \frac{2\alpha_k}{1+\alpha_k} f_k(z)$ and $\varepsilon_k = \max_{z \in [\alpha^p,1]} \left| \frac{\widetilde{f}_k(z) - z^{1/p}}{z^{1/p}} \right|$.  Then the conclusions~(\ref{claim1}) and~(\ref{equi}) hold with $f(z) = z^{1/p}$.  Furthermore, as $k \rightarrow \infty$, $\varepsilon_k \rightarrow 0$ monotonically with
\begin{equation} \label{epsestimatep}
\varepsilon_{k+1}  = C(m,\ell,p) \varepsilon_k^{m+\ell+1} + o(\varepsilon_k^{m+\ell+1}),
\end{equation}
where
\begin{equation} \label{constp}
C(m,\ell,p) = \frac{p^{m+\ell+1} m! \ell! (1/p)_{\ell+1} (1-1/p)_m }{ 2^{m+\ell} (m+\ell+1)!(m+\ell)! }.
\end{equation}
\end{theorem}

Note that when $p=2$ and $\ell \in \{m-1,m\}$,~(\ref{constp}) simplifies to $C(m,\ell,2) = 4^{-(m+\ell)}$.  This is consistent with the results of~\cite{gawlik2018zolotarev}, where it is shown that for these $m$, $\ell$, and $p$, an asymptotically sharp bound of the form $\varepsilon_k \le 4 \rho^{-(m+\ell+1)^k}$ holds with $\rho$ a constant depending on $\alpha$.

\subsection{Convergence of the matrix iteration}

An immediate consequence of Theorem~\ref{thm:pthroot} is that the iteration~(\ref{pthrootmit1}-\ref{pthrootmit2}) converges when $A$ is Hermitian positive definite with eigenvalues in $[\alpha^p,1]$.

\begin{corollary} \label{cor:spd}
Let $\alpha \in (0,1)$, $m,\ell \in \mathbb{N}_0$, and $p,n \in \mathbb{N}$ with $p \ge 2$ and $(m,\ell) \neq (0,0)$.  Let $A \in \mathbb{C}^{n \times n}$ be Hermitian positive definite.  If the eigenvalues of $A$ lie in $[\alpha^p,1]$, then the iteration~(\ref{pthrootmit1}-\ref{pthrootmit2}) generates a sequence $\widetilde{X}_k = 2\alpha_k X_k/(1+\alpha_k)$ that converges to $A^{1/p}$ with order $m+\ell+1$.  In particular, we have
\[
\|\widetilde{X}_k A^{-1/p} - I\|_2 \le \varepsilon_k,
\]
for every $k \ge 0$, where $\varepsilon_k$ obeys the recursion
\begin{equation} \label{epsitp}
\varepsilon_{k+1} = E_{m,\ell}\left(\mysqrt[p]{\cdot}, \left[\left(\frac{1-\varepsilon_k}{1+\varepsilon_k}\right)^p,1\right]\right) = C(m,\ell,p)\varepsilon_k^{m+\ell+1} + o(\varepsilon_k^{m+\ell+1}), \quad \varepsilon_0 = \frac{1-\alpha}{1+\alpha},
\end{equation}
and $C(m,\ell,p)$ is given by~(\ref{constp}).
\end{corollary}

A similar result holds for the coupled iteration~(\ref{pcoupled1}-\ref{pcoupled3}).

\begin{corollary} \label{cor:spdcoupled}
Let $\alpha,m,\ell,p,n$, and $A$ be as in Corollary~\ref{cor:spd}.  Then the coupled iteration~(\ref{pcoupled1}-\ref{pcoupled3}) generates sequences $\widetilde{Y}_k = (1+\alpha_k)^{p-1} Y_k/(2\alpha_k)^{p-1}$ and $\widetilde{Z}_k = (1+\alpha_k)Z_k/(2\alpha_k)$ that converge to $A^{1/p}$ and $A^{-1/p}$ respectively, with order $m+\ell+1$.  In particular, we have
\begin{align*}
\|\widetilde{Y}_k A^{-1/p} - I\|_2 &\le \frac{(1+\varepsilon_k)^{p-1}-1}{(1-\varepsilon_k)^{p-1}}, \\
\|\widetilde{Z}_k A^{1/p} - I\|_2 &\le \frac{\varepsilon_k}{1-\varepsilon_k},
\end{align*}
for every $k \ge 0$, where $\varepsilon_k$ obeys the recursion~(\ref{epsitp}).
\end{corollary}

Note that the bounds above imply corresponding bounds on the relative errors $\|\widetilde{X}_k - A^{1/p}\|_2 / \|A^{1/p}\|_2$, $\|\widetilde{Y}_k - A^{1/p}\|_2 / \|A^{1/p}\|_2$, and $\|\widetilde{Z}_k - A^{-1/p}\|_2 / \|A^{-1/p}\|_2$.  For instance,
\[
\frac{\|\widetilde{X}_k - A^{1/p}\|_2}{\|A^{1/p}\|_2} = \frac{\|(\widetilde{X}_k A^{-1/p} - I) A^{1/p}\|_2}{\|A^{1/p}\|_2} \le \|\widetilde{X}_k A^{-1/p} - I\|_2 \le \varepsilon_k.
\]

When $A$ is non-normal and/or has eigenvalues away from the positive real axis, the behavior of the matrix iteration~(\ref{pthrootmit1}-\ref{pthrootmit2}) (and hence~(\ref{pcoupled1}-\ref{pcoupled3}))  is dictated by the behavior of the scalar iteration~(\ref{pthrootit1}-\ref{pthrootit2}) on complex inputs $z$.  This has been analyzed in detail for the case $p=2$ in~\cite{gawlik2018backward}, but for $p>2$, numerical experiments indicate that the scalar iteration converges in a subset of the complex plane with fractal structure, 
a typical feature of iterations for the $p^{th}$ root.  We study this behavior numerically in Section~\ref{sec:numerical}.  It remains an open problem to determine theoretically the convergence region $\{z \in \mathbb{C} \mid \lim_{k \rightarrow \infty} f_k(z) = z^{1/p}\}$ for the iteration~(\ref{pthrootit1}-\ref{pthrootit2}).

\subsection{Special cases} \label{sec:specialcases}

For certain values of $m$, $\ell$, and $p$, the theory above recovers some known results from the literature.  We discuss these situations below.

\subsubsection{Square roots}
When $p=2$, $m \in \mathbb{N}$, and $\ell \in \{m-1,m\}$, a remarkable phenomenon occurs, allowing us to draw the connection between Theorem~\ref{thm:fit} and the results of~\cite{gawlik2018zolotarev} that we alluded to earlier.  For these $p$, $m$, and $\ell$, the function $\widetilde{f}_k(z)$ is a rational function of type $(m_k,\ell_k)$, where $(m_k,\ell_k)$ is given by~(\ref{mklk}).
In both the case $\ell=m-1$ and the case $\ell=m$, we have
\[
m_k+\ell_k = (m+\ell+1)^k-1,
\]
so~(\ref{equi}) implies that $\frac{\widetilde{f}_k(z) - f(z)}{f(z)}$ equioscillates $m_k+\ell_k+2$ times on $[f^{-1}(\alpha),1]$.  It follows from the theory of rational minimax approximation that $\widetilde{f}_k(z)$ is the best rational approximant of $\sqrt{z}$ of type $(m_k,\ell_k)$ on $[\alpha^2,1]$: 
\[
\widetilde{f}_k(z) = r_{m_k,\ell_k}(z,\alpha,\sqrt{\cdot}), \text{ if } p=2 \text{ and } \ell \in \{m-1,m\}.
\]
In particular,
\[
\varepsilon_k  = E_{m,\ell}(\sqrt{\cdot},[\alpha_k^2,1]) = E_{m_k,\ell_k}(\sqrt{\cdot},[\alpha^2,1]), \text{ if } \ell \in \{m-1,m\},
\]
for every $k \ge 1$.  This shows that Theorem~\ref{thm:fit} includes~\cite[Theorem 1]{gawlik2018zolotarev} as a special case.

\subsubsection{Low-order iterations}

When $p \ge 2$ is an integer and $(m,\ell) = (1,0)$ or $(0,1)$, we recover variants of another family of iterations.

\begin{proposition} \label{prop:loworder}
Let $p \ge 2$ be an integer and $\alpha \in (0,1)$.  We have
\begin{equation} \label{rhat10}
\hat{r}_{1,0}(z,\alpha,\sqrt[p]{\cdot}) = \frac{1}{p}\left((p-1)\mu + \frac{z}{\mu^{p-1}} \right), \quad \mu = \left( \frac{\alpha - \alpha^p}{(p-1)(1-\alpha)} \right)^{1/p}.
\end{equation}
and
\begin{equation} \label{rhat01}
\hat{r}_{0,1}(z,\alpha,\sqrt[p]{\cdot}) = \frac{p}{(p+1)\nu-\nu^{p+1}z}, \quad \nu = \left(\frac{(p+1)(1-\alpha)}{1-\alpha^{p+1}}\right)^{1/p}.
\end{equation}
\end{proposition}

Note that the formula~(\ref{rhat10}) for $\hat{r}_{1,0}(z,\alpha,\sqrt[p]{\cdot})$ appears in~\cite[Theorem 2]{meinardus1980optimal} and~\cite{king1971improved}; see also~\cite[Lemma 3.2]{guo2006schur} for a related result.

The preceding proposition shows that when $(m,\ell)=(1,0)$, the iteration~(\ref{pthrootmit1}-\ref{pthrootmit2}) reads
\begin{align*}
X_{k+1} &= \frac{1}{p}\left( (p-1)\mu_k X_k + (\mu_k X_k)^{1-p} A \right), & X_0 &= I, \\
\alpha_{k+1} &= \frac{p \alpha_k}{(p-1)\mu_k + \mu_k^{1-p}\alpha_k^p}, & \alpha_0 &= \alpha,
\end{align*}
where
\begin{equation} \label{muk}
\mu_k = \left( \frac{\alpha_k - \alpha_k^p}{(p-1)(1-\alpha_k)} \right)^{1/p}.
\end{equation}
This is a scaled variant of the popular Newton iteration~\cite[Equation 7.5]{higham2008functions} for the matrix $p^{th}$ root.  The scaling heuristic above is reminiscent of one proposed by Hoskins and Walton~\cite{hoskins1979faster}, but theirs is based on type-$(1,0)$ rational minimax approximants of $z^{(p-1)/p}$.

On the other hand, when $(m,\ell)=(0,1)$, the iteration~(\ref{pthrootmit1}-\ref{pthrootmit2}) reads
\begin{align*}
X_{k+1} &= p X_k \left( (p+1)\nu_k I - \nu_k^{p+1} X_k^{-p} A) \right)^{-1}, & X_0 &= I, \\
\alpha_{k+1} &= \frac{1}{p} \alpha_k \left((p+1)\nu_k - \nu_k^{p+1}\alpha_k^p\right), & \alpha_0 &= \alpha,
\end{align*}
where
\begin{equation} \label{nuk}
\nu_k = \left( \frac{(p+1)(1 - \alpha_k)}{1-\alpha_k^{p+1}} \right)^{1/p}.
\end{equation}
In terms of the matrix $Z_k = X_k^{-1}$, the iteration for $X_k$ becomes
\begin{align*}
Z_{k+1} &= \frac{1}{p}\left( (p+1)\nu_k Z_k - (\nu_k Z_k)^{p+1} A \right), & Z_0 &= I,  
\end{align*}
which is a scaled variant of the inverse Newton iteration~\cite[Equation (7.12)]{higham2008functions} for computing $A^{-1/p}$.

\subsubsection{Pad\'e iterations}

We recover one more family of iterations by considering the limit as $\alpha \uparrow 1$ in~(\ref{pthrootmit1}-\ref{pthrootmit2}).  

Below, we say that a family of rational functions $\{r_\alpha \in \mathcal{R}_{m,\ell} \mid \alpha \in (0,1)\}$ converges coefficientwise to $r_1 \in \mathcal{R}_{m,\ell}$ as $\alpha \uparrow 1$ if the coefficients of the polynomials in the numerator and denominator of $r_\alpha$, appropriately normalized, approach those of $r_1$ as $\alpha \uparrow 1$.

\begin{proposition} \label{prop:pade}
As $\alpha \uparrow 1$, $\hat{r}(z,\alpha,\mysqrt[p]{\cdot})$ converges coefficientwise to the type-$(m,\ell)$ Pad\'e approximant $P_{m,\ell,p}(z)$ of $z^{1/p}$ at $z=1$:
\begin{equation} \label{padep}
P_{m,\ell,p}(z) = \sum_{j=0}^m \frac{(-m)_j (-1/p-\ell)_j}{j!(-\ell-m)_j} (1-z)^j \bigg/ \sum_{j=0}^\ell \frac{(1/p)_j (1/p-m)_m (j-\ell-m)_m}{ j! (-\ell-m)_m (j+1/p-m)_m } (1-z)^j.
\end{equation}
\end{proposition}

It follows that the iteration~(\ref{pthrootmit1}-\ref{pthrootmit2}) reduces formally to
\begin{equation} \label{pade}
X_{k+1} = X_k P_{m,\ell,p}\left( X_k^{-p}A \right), \quad X_0 = I
\end{equation}
as $\alpha \uparrow 1$.  This is precisely the Pad\'e iteration for the matrix $p^{th}$ root studied by Laszkiewicz and Zi\k{e}tak~\cite[Equation (36)]{laszkiewicz2009pade}.  When $(m,\ell)=(1,1)$, it is the Halley iteration~\cite[p. 11]{iannazzo2008family},~\cite{guo2010newton}.  In terms of $Y_k = X_k^{1-p}A$ and $Z_k = X_k^{-1}$, the iteration~(\ref{pade}) reads
\begin{align} 
Y_{k+1} &= Y_k Q_{\ell,m,p}\left( Z_k Y_k \right)^{p-1}, & Y_0 &= A, \label{padecoupled1} \\
Z_{k+1} &= Q_{\ell,m,p}\left( Z_k Y_k \right) Z_k, & Z_0 &= I, \label{padecoupled2}
\end{align}
where $Q_{\ell,m,p}(z) = P_{m,\ell,p}(z)^{-1}$.

For later use, it will be convenient to define
\begin{align*}
\hat{r}_{m,\ell}(z,1,\mysqrt[p]{\cdot}) &:= P_{m,\ell,p}(z), \\
h_{\ell,m,p}(z,1) &:= Q_{\ell,m,p}(z).
\end{align*}
The Pad\'e iterations~(\ref{pade}) and~(\ref{padecoupled1}-\ref{padecoupled2}) are then simply the iterations obtained by setting $\alpha=1$ in the minimax iterations~(\ref{pthrootmit1}-\ref{pthrootmit2}) and~(\ref{pcoupled1}-\ref{pcoupled3}), respectively.

\subsection{Stability of the coupled matrix iteration}

As alluded to earlier, the uncoupled matrix iteration~(\ref{pthrootmit1}-\ref{pthrootmit2}) exhibits numerical instability, whereas the coupled iteration~(\ref{pcoupled1}-\ref{pcoupled3}) does not.  We justify the latter claim below.

We recall the following definition.  A matrix iteration $X_{k+1} = g(X_k)$ with fixed point $X_*$ is said to be \emph{stable} in a neighborhood of $X_*$ if the Fr\'echet derivative of $g$ at $X_*$ has bounded powers at $X_*$~\cite[Definition 4.17]{higham2008functions}.  That is, if $L_g(A,E)$ denotes the Fr\'echet derivative of $g$ at $A \in \mathbb{C}^{n\times n}$ in a direction $E \in \mathbb{C}^{n \times n}$, then there exists a constant $c>0$ such $\|G^j(E)\| \le c\|E\|$ for every $j$ and every $E \in \mathbb{C}^{n \times n}$, where $G(E) = L_g(X_*,E)$.

We first address the stability of the coupled Pad\'e iteration~(\ref{padecoupled1}-\ref{padecoupled2}). 

\begin{proposition} \label{prop:stabilitypade}
Let $m,\ell \in \mathbb{N}_0$ and $p,n \in \mathbb{N}$ with $(m,\ell) \neq (0,0)$ and $p \ge 2$.  
The Pad\'e iteration~(\ref{padecoupled1}-\ref{padecoupled2}) is stable in a neighborhood of $(B,B^{-1})$ for any $B \in \mathbb{C}^{n \times n}$.  In particular, with $g(Y,Z) = (YQ_{\ell,m,p}(YZ)^{p-1},Q_{\ell,m,p}(YZ)Z)$, we have
\[
L_g(B,B^{-1}; E,F) = \frac{1}{p} \left( E - (p-1)BFB, \; (p-1)F - B^{-1}EB^{-1} \right)
\]
for any $E,F \in \mathbb{C}^{n \times n}$, and $L_g(B,B^{-1};\cdot,\cdot)$ is idempotent.
\end{proposition}

Consider now the coupled minimax iteration~(\ref{pcoupled1}-\ref{pcoupled3}).  Theorem~\ref{thm:fit} established that $\alpha_k$ converges to $1$ in~(\ref{pcoupled3}).  We argue in Section~\ref{sec:numerical} that when $\alpha_k$ is close to 1, it is numerically prudent to set $\alpha_k$ (and all subsequent iterates) equal to 1, thereby reverting to the Pad\'e iteration~(\ref{padecoupled1}-\ref{padecoupled2}).  Since the latter iteration is stable, it follows that the aforementioned modification of~(\ref{pcoupled1}-\ref{pcoupled3}) is stable as well.

\section{Proofs} \label{sec:proofs}

In this section, we prove Theorems~\ref{thm:fit} and~\ref{thm:pthroot}, Corollaries~\ref{cor:spd} and~\ref{cor:spdcoupled}, and Propositions and~\ref{prop:loworder},~\ref{prop:pade}, and~\ref{prop:stabilitypade}.

\subsection{Proof of Theorem~\ref{thm:fit}}
\subsubsection{Equioscillation}

To prove the claims~(\ref{claim1}) and~(\ref{equi}) in Theorem~\ref{thm:fit}, we use an inductive argument. When $k=0$,~(\ref{equi}) holds since the relative error $\frac{\widetilde{f}_0(z)-f(z)}{f(z)}=\frac{2\alpha}{f(z)(1+\alpha)}-1$ decreases monotonically from $\frac{1-\alpha}{1+\alpha}$ to $-\frac{1-\alpha}{1+\alpha}$ as $z$ runs from $f^{-1}(\alpha)$ to $1$.  This shows also that $\varepsilon_0 = \frac{1-\alpha}{1+\alpha}$, so~(\ref{alphaeps}) holds when $k=0$.  
Next, we prove two lemmas in preparation for the inductive step.

\begin{lemma} \label{lemma:alphaupdate2}
Let $f : [0,1] \rightarrow [0,1]$ be a continuous, increasing bijection satisfying~(\ref{assumption1}).  Then the recurrence~(\ref{genit2}) is equivalent to
\begin{equation} \label{alphaupdate2}
\alpha_{k+1} = \frac{1-E_{m,\ell}(f, [f^{-1}(\alpha_k),1])}{1+E_{m,\ell}(f, [f^{-1}(\alpha_k),1])}, \quad \alpha_0 = \alpha.
\end{equation}
\end{lemma}
\begin{proof}
Since
\[
\min_{z \in [f^{-1}(\alpha),1]} \frac{r_{m,\ell}(z,\alpha,f)}{f(z)} = 1 - E_{m,\ell}(f,[f^{-1}(\alpha),1]),
\]
the defining property~(\ref{rhat}) of $\hat{r}_{m,\ell}(z,\alpha,f)$ implies that
\[
\hat{r}_{m,\ell}(z,\alpha,f) = \frac{r_{m,\ell}(z,\alpha,f)}{1 - E_{m,\ell}(f,[f^{-1}(\alpha),1])}.
\]
Also, the assumption~(\ref{assumption1}) implies that
\[
\frac{r_{m,\ell}(f^{-1}(\alpha),\alpha,f)}{f(f^{-1}(\alpha))} = \max_{z \in [f^{-1}(\alpha),1]} \frac{r_{m,\ell}(z,\alpha,f)}{f(z)} = 1 + E_{m,\ell}(f,[f^{-1}(\alpha),1]),
\]
so
\[
\frac{\hat{r}_{m,\ell}(f^{-1}(\alpha),\alpha,f)}{\alpha} = \frac{1 + E_{m,\ell}(f,[f^{-1}(\alpha),1])}{1 - E_{m,\ell}(f,[f^{-1}(\alpha),1])}.
\]
Since this holds for any $\alpha \in (0,1)$, it follows that the recurrence~(\ref{genit2}) is equivalent to~(\ref{alphaupdate2}).
\end{proof}

\begin{lemma} \label{lemma:composition}
Let $f : [0,1] \rightarrow [0,1]$ be a continuous, increasing bijection satisfying~(\ref{assumption1}).
Let $\alpha \in (0,1)$ and $m,\ell \in \mathbb{N}_0$.  Let $\widetilde{F}(z)$ be any continuous function on $[f^{-1}(\alpha),1]$ with the property that $\frac{\widetilde{F}(z)-f(z)}{f(z)}$ equioscillates $q$ times on $[f^{-1}(\alpha),1]$ and achieves its extrema $\pm \varepsilon$ at the endpoints, where $q \ge 2$ and $0<\varepsilon<1$.  Define
\begin{align*}
\alpha' &= \frac{1-\varepsilon}{1+\varepsilon}, \\
\alpha'' &= \frac{1-E_{m,\ell}(f,[f^{-1}(\alpha'),1])}{1+E_{m,\ell}(f,[f^{-1}(\alpha'),1])}, \\
F(z) &= \frac{1+\alpha'}{2\alpha'} \widetilde{F}(z), \\
H(z) &= \frac{2\alpha''}{1+\alpha''} F(z) \hat{r}_{m,\ell}\left( f^{-1}\left( \frac{f(z)}{F(z)} \right), \alpha', f \right).
\end{align*}
Then $\frac{H(z)-f(z)}{f(z)}$ equioscillates $(m+\ell+1)(q-1)+1$ times on $[f^{-1}(\alpha),1]$ with extrema $\pm E_{m,\ell}(f,[f^{-1}(\alpha'),1])$, and it achieves its extrema at the endpoints.
\end{lemma}
\begin{proof}
The assumed equioscillation of $\frac{\widetilde{F}(z)}{f(z)}-1$ on $[f^{-1}(\alpha),1]$ implies that the function $\frac{\widetilde{F}(f^{-1}(z))}{z}-1$ equioscillates $q$ times on $[\alpha,1]$ with extrema $\pm \varepsilon$.
If we now define
\[
S(z) = \frac{z (1-\varepsilon^2)}{\widetilde{F}(f^{-1}(z))},
\]
then we conclude that $S(z)-1$ equioscillates $q$ times on $[\alpha,1]$ with extrema $\frac{1-\varepsilon^2}{1 \pm \varepsilon} -1 = \mp \varepsilon$.  Moreover, it achieves its extrema at the endpoints by our assumptions on $\widetilde{F}$.

By the same reasoning as above, the function
\[
s_{m,\ell}(z,\alpha',f) = \frac{z (1-\varepsilon'^2)}{r_{m,\ell}(f^{-1}(z),\alpha',f)}, \quad \varepsilon' = E_{m,\ell}(f,[f^{-1}(\alpha'),1]),
\]
has the property that $s_{m,\ell}(z,\alpha',f)-1$ equioscillates $m+\ell+2$ times on $[\alpha',1]$ with extrema $\pm \varepsilon'$, and it achieves its extrema at the endpoints by the assumption~(\ref{assumption1}).

Consider now the function
\begin{equation} \label{scomps}
g(z) = s_{m,\ell}\left( \frac{S(z)}{1+\varepsilon}, \alpha', f \right).
\end{equation}
We claim that $g(z)-1$ equioscillates on $[\alpha,1]$ with extrema $\pm \varepsilon'$.  To see this, we make two observations.  First, as $z$ runs from $\alpha$ to $1$, $\frac{S(z)}{1+\varepsilon}$ runs from/to $\frac{1-\varepsilon}{1+\varepsilon} = \alpha'$ to/from $\frac{1+\varepsilon}{1+\varepsilon} = 1$ a total of $q-1$ times, achieving its extrema at the endpoints each time.  Second, each time $y = \frac{S(z)}{1+\varepsilon}$ runs from/to $\alpha'$ to/from $1$, $s_{m,\ell}(y,\alpha',f)-1$ equioscillates $m+\ell+2$ times with extrema $\pm \varepsilon'$.  By counting extrema, we conclude that the composition~(\ref{scomps}) (minus 1) equioscillates
\[
(m+\ell+2)(q-1) - (q-2) = (m+\ell+1)(q-1) + 1
\]
times on $[\alpha,1]$ with extrema $\pm \varepsilon'$.

Finally, consider the function
\[
h(z) = \frac{(1-\varepsilon'^2)}{g(f(z))}.
\]
In view of the equioscillation of~(\ref{scomps}), the function $h(z)-1$ equioscillates $(m+\ell+1)(q-1) + 1$ times on $[f^{-1}(\alpha),1]$ with extrema $\frac{1-\varepsilon'^2}{1 \pm \varepsilon'} -1 = \mp \varepsilon'$, and it achieves its extrema at the endpoints. We will complete the proof by showing that $h(z) = \frac{H(z)}{f(z)}$.  Using the fact that $1-\varepsilon' = \frac{2\alpha''}{1+\alpha''}$, $\widetilde{F}(z) = (1-\varepsilon)F(z)$, and $r_{m,\ell}(z,\alpha,f) = (1-\varepsilon')\hat{r}_{m,\ell}(z,\alpha,f)$, we have
\begin{align*}
 h(z)
 &= \frac{(1-\varepsilon'^2)}{ s_{m,\ell}\left( \frac{S(f(z))}{1+\varepsilon}, \alpha', f \right) } \\
&= \frac{r_{m,\ell}\left( f^{-1}\left( \frac{S(f(z))}{1+\varepsilon} \right), \alpha', f \right) }{ \frac{S(f(z))}{1+\varepsilon} } \\
&=  \frac{r_{m,\ell}\left( f^{-1}\left( \frac{f(z)(1-\varepsilon)}{\widetilde{F}(z)} \right), \alpha', f \right) }{ \frac{f(z)(1-\varepsilon)}{\widetilde{F}(z)} } \\
&= (1-\varepsilon') \frac{ F(z) \hat{r}_{m,\ell}\left( f^{-1}\left( \frac{f(z)}{F(z)} \right), \alpha', f \right) }{ f(z) } \\
&= \frac{H(z)}{f(z)}.
\end{align*}
\qed\end{proof}

\begin{remark} \label{remark:sector}
When $f(z)=z^{1/p}$, the function 
\[
s_{m,\ell}(z,\alpha',\mysqrt[p]{\cdot}) = \frac{z (1-\varepsilon'^2)}{r_{m,\ell}(z^p,\alpha',\mysqrt[p]{\cdot})}
\]
appearing in the proof above is a rational approximant of the \emph{sector function} $\mathrm{sect}_p(z) = z/(z^p)^{1/p}$.  In fact, the proof above reveals that on each of the segments $\{ z \in \mathbb{C} \mid e^{-2\pi i j/p} z \in [\alpha',1]\}$, $j=0,1,2,\dots,p-1$, the relative error
\[
\frac{s_{m,\ell}(z,\alpha',\mysqrt[p]{\cdot}) - \mathrm{sect}_p(z)}{\mathrm{sect}_p(z)} = e^{-2\pi i j/p} s_{m,\ell}(z,\alpha',\mysqrt[p]{\cdot}) - 1
\]
is real-valued and equioscillates $m+\ell+2$ times with extrema $\pm \varepsilon'$.  In particular, for $\ell \in \{m-1,m\}$, $s_{m,\ell}(z,\alpha',\sqrt{\cdot})$ is Zolotarev's type-$(2\ell+1,2m)$ best rational approximant of the sign function $\mathrm{sign}(z)=z/(z^2)^{1/2}$ on $[-1,-\alpha'] \cup [\alpha',1]$~\cite{nakatsukasa2016computing}.
\end{remark}

We are now ready to prove~(\ref{claim1}-\ref{equi}).  Suppose~(\ref{equi}) and~(\ref{alphaeps}) hold at step $k$ in the iteration~(\ref{pthrootit1}-\ref{pthrootit2}).  Then Lemma~\ref{lemma:composition} (applied with $\widetilde{F}=\widetilde{f}_k$, $\varepsilon=\varepsilon_k$, and $q=(m+\ell+1)^k+1$, so that $\alpha'=\alpha_k$ and $\alpha''=\alpha_{k+1}$) implies that~(\ref{equi}) and~(\ref{alphaeps}) hold at step $k+1$, so in fact they hold for all $k$.  It now follows immediately that~(\ref{err}) is equivalent to~(\ref{alphaupdate2}), which, in turn, is equivalent to~(\ref{genit2}) by Lemma~\ref{lemma:alphaupdate2}.  This completes the proof of~(\ref{claim1}-\ref{equi}).

\subsubsection{Convergence}
We now address the last claim~(\ref{convergence}) of Theorem~\ref{thm:fit}, which concerns the convergence of $\varepsilon_k$ to $0$ in the iteration
\begin{equation} \label{epsit}
\varepsilon_{k+1} = G(\varepsilon_k), \quad \varepsilon_0 = \frac{1-\alpha}{1+\alpha}, 
\end{equation}
with $\alpha \in (0,1)$,
\begin{equation} \label{Gfunc}
G(\varepsilon) = E_{m,\ell}\left(f,\left[f^{-1}\left(\frac{1-\varepsilon}{1+\varepsilon}\right),1\right]\right),
\end{equation}
and $(m,\ell) \neq (0,0)$.
\begin{lemma} \label{lemma:Gproperties}
Let $m,\ell \in \mathbb{N}_0$, and let $f : [0,1] \rightarrow [0,1]$ be a continuous, increasing bijection satisfying~(\ref{assumption1}).  If $(m,\ell) \neq (0,0)$, then $G$ is continuous, nonnegative, and nondecreasing on $(0,1)$.  Furthermore, $G(\varepsilon) < \varepsilon$ for every $\varepsilon \in (0,1)$.
\end{lemma}
\begin{proof}
It is obvious that $G$ is nonnegative and nondecreasing.
To show that $G(\varepsilon) < \varepsilon$ for every $\varepsilon \in (0,1)$, note that~(\ref{Gfunc}) is no larger than the uniform relative error committed by the constant function $g(z)=1-\varepsilon$:
\[
-\varepsilon = \frac{1-\varepsilon - f(1)}{f(1)} \le \frac{g(z)-f(z)}{f(z)} \le \frac{ 1-\varepsilon - f\left(f^{-1}\left( \frac{1-\varepsilon}{1+\varepsilon} \right)\right) }{f\left(f^{-1}\left( \frac{1-\varepsilon}{1+\varepsilon} \right)\right)} = \varepsilon
\]
for every $z \in \left[f^{-1}\left(\frac{1-\varepsilon}{1+\varepsilon}\right),1\right]$.
This establishes that $G(\varepsilon) \le \varepsilon$.  The inequality is in fact strict since we assumed~(\ref{assumption1}), which implies that the minimizer of the relative error is not a constant function when $(m,\ell) \neq (0,0)$.  It remains to show that $G$ is continuous on $(0,1)$.  We assumed in~(\ref{assumption1}) that the minimizer for $E_{m,\ell}(f,[f^{-1}(\alpha),1])$ has defect 0 in $\mathcal{R}_{m,\ell}$ for each $\alpha \in (0,1)$, so, for each fixed $\alpha \in (0,1)$, the map $g \mapsto r_{m,\ell}(\cdot,\alpha,g)$ is continuous with respect to the uniform norm at $g=f$~\cite{maehly1960tschebyscheff}.  By considering functions $g$ obtained by scaling and translating the input to $f$, we deduce that $r_{m,\ell}(\cdot,\alpha,f)$ depends continuously on $\alpha \in (0,1)$, again with respect to the uniform norm.
Hence, the map $\alpha \mapsto E_{m,\ell}(f,[f^{-1}(\alpha),1])$ is continuous on $(0,1)$, and so too is $G$.
\qed\end{proof}

It follows from the above properties of $G$ that $\varepsilon_k \rightarrow 0$ monotonically in the iteration $\varepsilon_{k+1} = G(\varepsilon_k)$ for every $\varepsilon_0 \in (0,1)$.

\subsubsection{Rate of convergence}
It remains to show that the order of convergence of $\varepsilon_k$ to $0$ is $m+\ell+1$.  As we explained in the paragraph below Theorem~\ref{thm:fit}, it suffices to note that when $f$ is $C^{m+\ell+1}$ in a neighborhood of $1$,
\begin{equation*} 
E_{m,\ell}(f,[a,1]) = O((1-a)^{m+\ell+1}), \, \text{ as } a \rightarrow 1.
\end{equation*}
Indeed, this, together with~(\ref{err}), gives
\begin{equation} \label{orderml1}
\varepsilon_{k+1} = O\left( \left( 1 - f^{-1}\left(\frac{1-\varepsilon_k}{1+\varepsilon_k}\right) \right)^{m+\ell+1}\right) = O(\varepsilon_k^{m+\ell+1}),
\end{equation}
assuming $f^{-1}$ is Lipschitz near $1$ and $f^{-1}(1)=1$.  Below, we give more precise information about the constant implicit in~(\ref{orderml1}).  We begin with a lemma that shows, in essence, that the uniform error in the best type-$(m,\ell)$ rational approximant of a function $g(z)$ on a small interval $[-\delta,\delta]$ is about $2^{m+\ell}$ times smaller than the uniform error in the type-$(m,\ell)$ Pad\'e approximant of $g(z)$.  
\begin{lemma} \label{lemma:errshape}
Let $g(z)$ be $C^{m+\ell+1}$ and positive in a neighborhood of $0$.  Assume that the type-$(m,\ell)$ Pad\'e approximant $p(z)$ of $g(z)$ about $0$ has defect $0$ in $\mathcal{R}_{m,\ell}$, and 
\[
p(z) - g(z) = c_g z^{m+\ell+1}+o(z^{m+\ell+1}),
\]
where $c_g \in \mathbb{R}$.    For each $\delta>0$, let 
\[
r_\delta = \argmin_{r \in \mathcal{R}_{m,\ell}} \max_{-\delta \le z \le \delta} \left| \frac{r(z)-g(z)}{g(z)} \right|.
\]
Then, as $\delta \rightarrow 0$,
\[
\max_{-\delta \le z \le \delta} \left| \frac{r_\delta(z) - g(z)}{g(z)} \right| = \frac{2 |c_g|}{g(0)} \left(\frac{\delta}{2}\right)^{m+\ell+1} + o(\delta^{m+\ell+1}).
\]
\end{lemma}
\begin{proof}
Let 
\begin{equation} \label{bestpoly}
q = \argmin_{r \in \mathcal{R}_{m+\ell,0}} \max_{-\delta \le z \le \delta} |r(z)-z^{m+\ell+1}|.
\end{equation}
Among polynomials of degree $m+\ell+1$ with unit leading coefficient, the polynomial $z^{m+\ell+1}-q(z)$ is the one that deviates least from $0$ on $[-\delta,\delta]$.  
Up to a rescaling, this is precisely the degree-$(m+\ell+1)$ Chebyshev polynomial of the first kind $T_{m+\ell+1}(z)$:
\[
z^{m+\ell+1}-q(z) = 2 \left(\frac{\delta}{2}\right)^{m+\ell+1} T_{m+\ell+1}\left(\frac{z}{\delta}\right).
\]
Now let $R(z)$ be the type $(m,\ell)$-Pad\'e approximant of 
\[
\bar{g}(z) = g(z) - c_g q(z).
\]
Since we assumed that the Pad\'e approximant of $g(z)$ has defect $0$ in $\mathcal{R}_{m,\ell}$, the Taylor coefficients of $R(z)$ approach those of $p(z)$ as $\delta \rightarrow 0$~\cite[Corollary of Theorem 2a]{trefethen1985convergence}.  It follows that for each $\delta>0$ sufficiently small,
\[
R(z) - \bar{g}(z) = \bar{c}_g z^{m+\ell+1} + o(z^{m+\ell+1}),
\]
for some $\bar{c}_g$ with $\bar{c}_g - c_g = o(1)$ as $\delta \rightarrow 0$.
Thus, for each $\delta>0$ sufficiently small,
\begin{align*}
R(z) - g(z) 
&= R(z) - \bar{g}(z) - c_g q(z) \\
&= \bar{c}_g z^{m+\ell+1} - c_g z^{m+\ell+1} +  2 c_g \left(\frac{\delta}{2}\right)^{m+\ell+1} T_{m+\ell+1}\left(\frac{z}{\delta}\right) + o(z^{m+\ell+1}). 
\end{align*}
Hence, as $\delta \rightarrow 0$,
\[ 
R(z) - g(z) = 2 c_g \left(\frac{\delta}{2}\right)^{m+\ell+1} T_{m+\ell+1}\left(\frac{z}{\delta}\right) + o(\delta^{m+\ell+1})
\]
for every $z \in [-\delta,\delta]$, uniformly in $z$.  Multiplying by $\frac{1}{g(z)} = \frac{1}{g(0)} + o(1)$, we conclude that
\begin{equation} \label{errR}
\frac{R(z)-g(z)}{g(z)} = \frac{2c_g}{g(0)} \left(\frac{\delta}{2}\right)^{m+\ell+1} T_{m+\ell+1}\left(\frac{z}{\delta}\right) + o(\delta^{m+\ell+1})
\end{equation}
for every $z \in [-\delta,\delta]$, uniformly in $z$.
Finally, by the definition of $r_\delta$,
\[
\max_{-\delta \le z \le \delta} \left|\frac{r_\delta(z)-g(z)}{g(z)}\right| \le \max_{-\delta \le z \le \delta} \left|\frac{R(z)-g(z)}{g(z)}\right| = \frac{2c_g}{g(0)} \left(\frac{\delta}{2}\right)^{m+\ell+1} + o(\delta^{m+\ell+1}).
\]
In fact, this bound is sharp, for the following reason.  The relation~(\ref{errR}) shows that for $\delta$ sufficiently small, $\frac{R(z)-g(z)}{g(z)}$ approximately equioscillates, in the sense that there exist $m+\ell+2$ points $-\delta \le z_0 \le z_1 \le \dots \le z_{m+\ell+1} \le \delta$ at which $\frac{R(z)-g(z)}{g(z)}$ alternates in sign and satisfies
\[
\left| \frac{R(z_j)-g(z_j)}{g(z_j)} \right| \ge \frac{2|c_g|}{g(0)} \left(\frac{\delta}{2}\right)^{m+\ell+1} - \gamma, \quad j=0,1,\dots,m+\ell+1,
\]
where $\gamma =o(\delta^{m+\ell+1})$.
The de la Vall\'ee Poussin lower bound~\cite[Exercise 24.5]{trefethen2013approximation} then implies that 
\[
\max_{-\delta \le z \le \delta} \left|\frac{r_\delta(z)-g(z)}{g(z)}\right| \ge \frac{2|c_g|}{g(0)} \left(\frac{\delta}{2}\right)^{m+\ell+1} - \gamma.
\]
\qed\end{proof}

\begin{remark}
The proof above suggests a heuristic for constructing near-best rational minimax approximants on short intervals $[-\delta,\delta]$: one computes the Pad\'e approximant of $\bar{g}(z) = g(z) - c_g z^{m+\ell+1} + 2c_g (\delta/2)^{m+\ell+1} T_{m+\ell+1}(z/\delta)$ rather than $g(z)$.
\end{remark}
\begin{remark}
The near equioscillation of $R$ in the proof above can be used to show that $R$ is close to $r_\delta$: $R(z)-r_\delta(z) = o(\delta^{m+\ell+1})$, uniformly in $z \in [-\delta,\delta]$ as $\delta \rightarrow 0$.  The argument is essentially the same as the one used in~\cite[p. 429-430]{trefethen1983caratheodory} to show that Carath\'eodory-F\'ejer approximants are close to minimax approximants on small intervals.
\end{remark}

It is now a simple matter to estimate the constant implicit in~(\ref{orderml1}). 
As $\varepsilon \rightarrow 0$, the above lemma gives
\begin{align*}
G(\varepsilon) &= E_{m,\ell}\left(f,\left[f^{-1}\left(\frac{1-\varepsilon}{1+\varepsilon}\right),1\right]\right) \\
&= \max_{f^{-1}\left(\frac{1-\varepsilon}{1+\varepsilon}\right) \le z \le 1} \left| \frac{r_{m,\ell}(z,\alpha,f)-f(z)}{f(z)} \right| \\
&= \frac{2 |c_{f,\delta}|}{f(1-\delta)} \left( \frac{\delta}{2} \right)^{m+\ell+1} + o(\delta^{m+\ell+1}), \\
&= 2 |c_{f,\delta}| \left( \frac{\delta}{2} \right)^{m+\ell+1} + o(\delta^{m+\ell+1}),
\end{align*}
where 
\[
\delta = \frac{1}{2}\left( 1 - f^{-1}(\alpha) \right), \quad \alpha = \frac{1-\varepsilon}{1+\varepsilon},
\]
and $c_{f,\delta}$ is the Taylor coefficient of $(z-1+\delta)^{m+\ell+1}$ in the difference between $f(z)$ and its type-$(m,\ell)$ Pad\'e approximant about $z=1-\delta$.  
A short calculation shows that $\delta = \varepsilon (f^{-1})'(1)  + o(\varepsilon) = \varepsilon/f'(1)+o(\varepsilon)$ and $c_f := c_{f,0} = c_{f,\delta} + o(1)$, so 
\[
G(\varepsilon) = \frac{|c_f|}{2^{m+\ell} f'(1)^{m+\ell+1}}  \varepsilon^{m+\ell+1} + o(\varepsilon^{m+\ell+1}).
\]
It follows that in the iteration~(\ref{epsit}), we have
\begin{equation} \label{epsestimate}
\varepsilon_{k+1}  = \frac{|c_f|}{2^{m+\ell} f'(1)^{m+\ell+1}}  \varepsilon_k^{m+\ell+1} + o(\varepsilon_k^{m+\ell+1}).
\end{equation}

\subsection{Proof of Theorem~\ref{thm:pthroot}}

Having proved Theorem~\ref{thm:fit}, we now verify that the function $f(z) = z^{1/p}$ satisfies the hypothesis~(\ref{assumption1}), and we prove Theorem~\ref{thm:pthroot}.
 
We begin by establishing a few properties of the minimax approximants $r_{m,\ell}(z,\alpha,\mysqrt[p]{\cdot})$.
The proof of the following lemma is similar to that in~\cite[Lemma 2]{stahl2003best}, which studies rational functions of type $(\ell+1,\ell)$ that minimize the maximum \emph{absolute} error on $[0,1]$ rather than the maximum \emph{relative} error on $[\alpha,1]$, $\alpha>0$.  The proof makes use of the following terminology.  A \emph{Chebyshev system} of dimension $N$ on an interval $I \subseteq \mathbb{R}$ is a linearly independent set $\{g_j(z)\}_{j=1}^N$ of continuous functions on $I$ with the property that any nontrivial linear combination $\sum_{j=1}^N c_j g_j(z)$ has at most $N-1$ (distinct) roots in $I$.

\begin{lemma} \label{lemma:details}
Let $m,\ell \in \mathbb{N}_0$, $0 < a < b < \infty$, and $p \in \mathbb{N}$, $p \ge 2$.  If $r \in \mathcal{R}_{m,\ell}$ minimizes 
\[
\max_{z \in [a,b]} |e(z)|, \quad e(z) =  \frac{r(z) - z^{1/p}}{z^{1/p}},
\]
then $r$ has exact type $(m,\ell)$, $e(z)$ equioscillates exactly $m+\ell+2$ times on $[a,b]$, and
\begin{align} 
e(a) &= \max_{z \in [a,b]} |e(z)|, \label{endpoint} \\
e(b) &= (-1)^{m+\ell+1} \max_{z \in [a,b]} |e(z)|. \label{endpointb}
\end{align}
\end{lemma}
\begin{proof}
Suppose that $r(z)=g(z)/h(z)$, where $g(z)$ and $h(z)$ are polynomials of exact degree $m' \le m$ and $\ell' \le \ell$, respectively.  Observe that the function 
\[
z^{1/p} h(z)e(z) = g(z) - z^{1/p} h(z)
\]
belongs to the space $W$ spanned by 
\[
\{ 1, z, z^2, \dots, z^{m'}, z^{1/p}, z^{1+1/p}, z^{2+1/p}, \dots, z^{\ell'+1/p} \},
\]
which is a Chebyshev system on $[a,b]$ of dimension $m'+\ell'+2$.  Thus, $z^{1/p} h(z) e(z)$ has at most $m'+\ell'+1$ zeros on $[a,b]$.  In particular, $e(z)$ has at most $m'+\ell'+1$ zeros on $[a,b]$, so it equioscillates at most $m'+\ell'+2$ times on $[a,b]$.  But $e(z)$ equioscillates at least $m+\ell+2-d$ times on $[a,b]$, where $d=\min\{m-m',\ell-\ell'\} \ge 0$.  It follows that
\[
m'+\ell'+2 \ge m+\ell+2-d,
\]
so
\[
d \ge (m-m') + (\ell-\ell') \ge 2d.
\]
From this we conclude that $d=0$, $m'=m$, $\ell'=\ell$, and $e(z)$ equioscillates exactly $m+\ell+2$ times on $[a,b]$.  

Let $a \le z_0 < z_1 < \dots < z_{m+\ell+1} \le b$ be the points at which $e(z)$ achieves its extrema on $[a,b]$.  Suppose that $z_0 > a$ or $z_{m+\ell+1} < b$.  By considering the graph of $e(z)$, one easily deduces that there exists $c \in \mathbb{R}$ such that $e(z)-c$ has at least $m+\ell+2$ roots in $[a,b]$.  But
\[
z^{1/p}h(z)(e(z)-c) = z^{1/p}h(z)e(z) - c z^{1/p} h(z) \in W,
\]
so $z^{1/p}h(z)(e(z)-c)$ has at most $m'+\ell'+1 = m+\ell+1$ roots in $[a,b]$.  In particular, $e(z)-c$ has at most $m+\ell+1$ roots in $[a,b]$, a contradiction.  It follows that $z_0=a$ and $z_{m+\ell+1}=b$.

It remains to verify that the signs in~(\ref{endpoint}-\ref{endpointb}) are correct.  Consider the dependence of $e(z)$ on the parameters $a$ and $b$.  Denote this dependence by $e(z;a,b)$.  By an argument similar to the one made in the proof of Lemma~\ref{lemma:Gproperties}, the maps $a \mapsto e(a;a,b)$ and $b \mapsto e(a;a,b)$ are continuous on $(0,b)$ and $(a,\infty)$, respectively.  These maps also have no zeros, since $e(z;a,b)$ has a nonzero extremum at $z=a$ for every $0<a<b<\infty$.   Now, for small $\delta>0$, the proof of Lemma~\ref{lemma:errshape} shows that for $z \in [1-\delta,1+\delta]$, 
\[
e(z;1-\delta,1+\delta) = 2c_f \left(\frac{\delta}{2}\right)^{m+\ell+1} T_{m+\ell+1}\left(\frac{z-1}{\delta}\right) + o(\delta^{m+\ell+1}),
\]
where $c_f$ is the coefficient of $(z-1)^{m+\ell+1}$ in the Taylor expansion of $P_{m,\ell,p}(z)-z^{1/p}$ about $z=1$.
In particular, $e(1-\delta;1-\delta,1+\delta)$ has the same sign as $c_f T_{m+\ell+1}(-1)=(-1)^{m+\ell+1}c_f$ for $\delta$ close to $0$, which, as we verify below in~(\ref{signcf}), is positive.  By continuity, $e(a;a,b)>0$ for every $0<a<b<\infty$, and~(\ref{endpoint}-\ref{endpointb}) follow. 
\end{proof}

The preceding lemma shows that the function $f(z) = z^{1/p}$ satisfies the hypothesis~(\ref{assumption1}), so Theorem~\ref{thm:pthroot} will follow if we can show that the constant $C(m,\ell,p)$ in the estimate~(\ref{epsestimatep}) is given by~(\ref{constp}).  In view of the general estimate~(\ref{epsestimate}), it suffices to determine the coefficient $c_f$ of the leading-order term $c_f (z-1)^{m+\ell+1}$ in $P_{m,\ell,p}(z)-z^{1/p}$, where $P_{m,\ell,p}(z)$ is the Pad\'e approximant~(\ref{padep}) of $z^{1/p}$ about $z=1$.  
This is given by~\cite[Lemma 3.12]{gomilko2012regions}
\begin{equation} \label{cfsum}
c_f = (-1)^{m+\ell+1} \frac{ m! \ell! (1/p)_{\ell+1} (1-1/p)_m }{ (m+\ell+1)!(m+\ell)! }.
\end{equation} 
Inserting this into~(\ref{epsestimate}) and noting that $f'(1)=\frac{1}{p}$ and 
\begin{equation} \label{signcf}
|c_f| = (-1)^{m+\ell+1}c_f,
\end{equation}
we obtain~(\ref{constp}).

\subsection{Proof of Corollaries~\ref{cor:spd} and~\ref{cor:spdcoupled}}

To prove Corollaries~\ref{cor:spd} and~\ref{cor:spdcoupled}, observe that with $e_k(z) = \frac{\widetilde{f}_k(z)-z^{1/p}}{z^{1/p}}$, we have
\begin{align*}
\widetilde{X}_k A^{-1/p} - I &= e_k(A), 
\end{align*}
\begin{align*}
\widetilde{Y}_k A^{-1/p} - I 
&= \widetilde{X}_k^{-(p-1)} A^{(p-1)/p} - I  \\
&= (I+e_k(A))^{-(p-1)} \left( I - (I+e_k(A))^{p-1} \right),
\end{align*}
and
\begin{align*}
\widetilde{Z}_k A^{1/p} - I 
&= \widetilde{X}_k^{-1} A^{1/p} - I \\
&= -(I+e_k(A))^{-1} e_k(A).
\end{align*}
The results follow from the above equalities and the bounds
\[
\|e_k(A)\|_2 \le \max_{\alpha^p \le z \le 1} |e_k(z)| = \varepsilon_k,
\]
\[
\|(I+e_k(A))^{-1}\|_2 \le \frac{1}{1-\|e_k(A)\|_2} \le \frac{1}{1-\varepsilon_k},
\]
and
\begin{align*}
\| I - (I+e_k(A))^{p-1} \|_2  
&= \left\| -\sum_{j=1}^{p-1} \binom{p-1}{j} e_k(A)^j \right\|_2 \\
&\le \sum_{j=1}^{p-1} \binom{p-1}{j} \varepsilon_k^j \\
&= (1+\varepsilon_k)^{p-1} - 1.
\end{align*}

\subsection{Proof of Proposition~\ref{prop:loworder}}

To prove the formula~(\ref{rhat10}) for $\hat{r}_{1,0}(z,\alpha,\mysqrt[p]{\cdot})$, it suffices to show that the function
\[
\hat{e}(z) := \frac{\hat{r}_{1,0}(z,\alpha,\mysqrt[p]{\cdot}) - z^{1/p}}{z^{1/p}}
\]
achieves its global maximum on $[\alpha^p,1]$ at both endpoints and has global minimum $0$ on $[\alpha^p,1]$.  Indeed, if this is the case, then the rescaled function
\[
\frac{2}{2+\hat{e}(1)} \hat{r}_{1,0}(z,\alpha,\mysqrt[p]{\cdot}) 
\]
has relative error which equioscillates three times on $[\alpha^p,1]$, and so must be the minimizer for $E_{1,0}(\mysqrt[p]{\cdot},[\alpha^p,1])$.  A calculation verifies that $\hat{e}(z)$ has a critical point at $z=\mu^p$, $\hat{e}(\mu^p)=0$, $\hat{e}(\alpha^p)=\hat{e}(1)$, $\hat{e}(z)$ is decreasing on $(\alpha^p,\mu^p)$, and  $\hat{e}(z)$ is increasing on $(\mu^p,1)$.

The proof of~(\ref{rhat01}) is similar. In this case, a calculation verifies that the function
\[
\hat{e}(z) := \frac{\hat{r}_{0,1}(z,\alpha,\mysqrt[p]{\cdot}) - z^{1/p}}{z^{1/p}}
\] 
has a critical point at $z=1/\nu^p$, $\hat{e}(1/\nu^p)=0$, $\hat{e}(\alpha^p)=\hat{e}(1)$, $\hat{e}(z)$ is decreasing on $(\alpha^p,1/\nu^p)$, and  $\hat{e}(z)$ is increasing on $(1/\nu^p,1)$.

\subsection{Proof of Proposition~\ref{prop:pade}}

Trefethen and Gutknecht~\cite[Theorem 3b]{trefethen1985convergence} have shown that for any function $f$ analytic in a neighborhood of $1$, $\argmin_{r \in \mathcal{R}_{m,\ell}} \linebreak \max_{z \in [1-\delta,1]} |r(z)-f(z)|$ converges coefficientwise as $\delta \rightarrow 0$ to the type-$(m,\ell)$ Pad\'e approximant of $f$ about $z=1$, provided that the Pad\'e approximant has defect $0$ in $\mathcal{R}_{m,\ell}$.  Their proof carries over easily to minimizers of the relative error $|(r(z)-f(z))/f(z)|$, assuming $f(1) \neq 0$.  Since $P_{m,\ell,p}(z)$ has defect $0$ in $\mathcal{R}_{m,\ell}$~\cite{gomilko2012pade}, Proposition~\ref{prop:pade} follows.  The explicit formula~(\ref{padep}) for $P_{m,\ell,p}(z)$ is from~\cite[p. 954]{laszkiewicz2009pade}.

\subsection{Proof of Proposition~\ref{prop:stabilitypade}}

Since $Q_{\ell,m,p}(z)^{-1} = P_{m,\ell,p}(z)$ is a Pad\'e approximant of $f(z)=z^{1/p}$ about $z=1$ of type $(m,\ell) \neq (0,0)$, we have $Q_{\ell,m,p}(1)=1$ and 
\[
-Q_{\ell,m,p}'(1) = \frac{-Q_{\ell,m,p}'(1)}{Q_{\ell,m,p}(1)^2} = P_{m,\ell,p}'(1) = f'(1) = \frac{1}{p}.
\]
Hence, $Q_{\ell,m,p}(I) = I$, $L_{Q_{\ell,m,p}}(I,E) = -\frac{1}{p}E$, and $L_{Q_{\ell,m,p}^{p-1}}(I,E) = -\frac{p-1}{p}E$ for any $E \in \mathbb{C}^{n \times n}$.  Thus, with $g(Y,Z) = (Y Q_{\ell,m,p}(ZY)^{p-1}, Q_{\ell,m,p}(ZY) Z)$, we obtain
\begin{align*}
L_g(B,B^{-1}; E,F) &= \left( E - B \left(\frac{p-1}{p}\right) (FB+B^{-1}E), \; F - \frac{1}{p} (FB+B^{-1}E) B^{-1} \right)\\
&= \frac{1}{p} \left( E - (p-1)BFB, \; (p-1)F - B^{-1}EB^{-1} \right).
\end{align*}
Setting $\widetilde{E} = \frac{1}{p}(E-(p-1)BFB)$ and $\widetilde{F} = \frac{1}{p}((p-1)F-B^{-1}EB^{-1})$, we find that $L_g(B,B^{-1}; \widetilde{E},\widetilde{F}) = L_g(B,B^{-1}; E, F)$, so $L_g(B,B^{-1}; \cdot,\cdot)$ is idempotent.

\section{Numerical examples} \label{sec:numerical}

In this section, we present numerical examples and discuss the implementation of the rational minimax iteration~(\ref{pcoupled1}-\ref{pcoupled3}).

\subsection{Implementation}
Implementing the rational minimax iteration~(\ref{pcoupled1}-\ref{pcoupled3}) requires evaluating the rational function $h_{\ell,m,p}(z,\alpha_k) = \hat{r}_{m,\ell}(z,\alpha_k,\mysqrt[p]{\cdot})^{-1}$ at a matrix argument $Z_k Y_k$.  With the exception of the special cases detailed in Section~\ref{sec:specialcases}, explicit formulas for this function are not available.  Nevertheless, $\hat{r}_{m,\ell}(z,\alpha_k,\mysqrt[p]{\cdot})$ (or, more precisely, its unscaled counterpart $r_{m,\ell}(z,\alpha_k,\mysqrt[p]{\cdot})$) can be computed numerically using, for instance, the function \verb$MiniMaxApproximation$ from Mathematica's \verb$FunctionApproximations$ package.  We used this function along with \verb$Apart$ to compute $h_{\ell,m,p}(z,\alpha_k)$ in partial fraction form.  For $\alpha_k$ close to $1$, the computation of $h_{\ell,m,p}(z,\alpha_k)$ poses numerical difficulties, so we rounded $\alpha_k$ to $1$ (thereby reverting to the Pad\'e iteration~(\ref{padecoupled1}-\ref{padecoupled2})) whenever $\alpha_k>0.99$.  We also observed that for $\alpha_k$ close to $0$ and $\ell=m$, accuracy improved if $r_{m,m}(z,\alpha_k,\mysqrt[p]{\cdot})$ was computed as $R(1/z)$, where $R = \argmin_{r \in \mathcal{R}_{m,m}} \max_{1 \le z \le \alpha_k^{-p}} |(r(z)-z^{-1/p})/z^{-1/p}|$. 

Note that a more robust option for computing minimizers of the maximum \emph{absolute} error $|r(z)-f(z)|$ is the Chebfun function \verb$minimax$~\cite{driscoll2014chebfun}.  However, Chebfun currently does not support minimization of the maximum \emph{relative} error $|(r(z)-f(z))/f(z)|$.
 
Algorithm~\ref{alg:minimax} summarizes the implementation of the rational minimax iteration~(\ref{pcoupled1}-\ref{pcoupled3}).  For simplicity, it focuses on the type $(m,m)$ iteration.  The type $(m,\ell)$ iteration with $\ell \neq m$ is similar, but the form of the partial fraction expansion of $h_{\ell,m,p}(z,\alpha)$ varies with $\ell$.  In the algorithm, the eigenvalues of $A$ with the smallest and largest magnitudes are denoted $\lambda_{\mathrm{min}}(A)$ and $\lambda_{\mathrm{max}}(A)$,  respectively.

\begin{algorithm}
\caption{Type-$(m,m)$ rational minimax iteration for the matrix $p^{th}$ root}
\label{alg:minimax}
\begin{algorithmic}[1]
\STATE{$\tau = |\lambda_{\mathrm{max}}(A)|$} \label{line:scaling}
\STATE{$\alpha_0 = |\lambda_{\mathrm{min}}(A)/\lambda_{\mathrm{max}}(A)|^{1/p}$} \label{line:alpha0}
\STATE{$Y_0 = A/\tau$}
\STATE{$Z_0 = I$}
\STATE{$k=0$}
\WHILE{not converged}
\STATE{Compute $h_{m,m,p}(z,\alpha_k)$ and its partial fraction expansion
\[
h_{m,m,p}(z,\alpha_k) = a_0 + \sum_{j=1}^m \frac{a_j}{z+b_j}.
\]
}
\STATE{$W = \sum_{j=1}^m a_j (Z_k Y_k + b_j I)^{-1}$} \label{line:sum}
\STATE{$Y_{k+1} = Y_k(a_0 I + W)^{p-1}$} \label{line:Ykp1}
\STATE{$Z_{k+1} = a_0 Z_k + W Z_k$} \label{line:Zkp1}
\STATE{$\alpha_{k+1} = \alpha_k h_{m,m,p}(\alpha_k^p,\alpha_k)$}
\STATE{$k = k+1$}
\ENDWHILE
\STATE{$\widetilde{Y}_k = \tau^{1/p} (1+\alpha_k)^{p-1} Y_k / (2\alpha_k)^{p-1}$}
\STATE{$\widetilde{Z}_k = \tau^{-1/p} (1+\alpha_k) Z_k / (2\alpha_k)$}
\RETURN $\widetilde{Y}_k \approx A^{1/p}$, $\widetilde{Z}_k \approx A^{-1/p}$
\end{algorithmic} 
\end{algorithm}

The choices of $\alpha_0$ and $\tau$ used in the algorithm are motivated by Corollary~\ref{cor:spdcoupled}: they ensure that the spectrum of $A/\tau$ is contained in the annulus $\{z \in \mathbb{C} \mid \alpha_0^p \le |z| \le 1\}$.  In particular, if $A$ is Hermitian positive definite, then the spectrum of $A/\tau$ is contained in $[\alpha_0^p,1]$, and Corollary~\ref{cor:spdcoupled} is directly applicable.  Neither $\lambda_{\mathrm{min}}(A)$ nor $\lambda_{\mathrm{max}}(A)$ need to be computed accurately; our experience suggests that estimates can be used without significantly degrading the algorithm's performance.

As a termination criterion, we terminated the iterations when
\[
\|\widetilde{Z}_{k-1} \widetilde{Y}_{k-1} - I \|_{\infty} \le p \left( \frac{\Delta}{(p-1)C(m,\ell,p)} \right)^{1/(m+\ell+1)},
\]
where $\Delta=10^{-15}$ is a relative error tolerance.
This is a generalization to arbitrary $p$ of the termination criterion described in~\cite[Section 4.3]{gawlik2018zolotarev}. 

\paragraph{Floating point operations}
If $A$ is $n \times n$ and $(a_0 I +W)^{p-1}$ is computed with binary powering in Line~\ref{line:Ykp1} of Algorithm~\ref{alg:minimax}, then the cost of each iteration in Algorithm~\ref{alg:minimax} is about $(6+2m+\beta \log_2(p-1))n^3$ flops, where $\beta \in [1,2]$~\cite[p. 72]{higham2008functions}.  In the first iteration, the cost reduces to $(2+2m+\beta \log_2(p-1))n^3$ flops since $Z_0=I$.  If parallelism is exploited, then the $m$ matrix inversions in Line~\ref{line:sum} can be performed simultaneously, as can Lines~\ref{line:Ykp1}-\ref{line:Zkp1}.  The effective cost of such a parallel implementation is $(4+\beta \log_2(p-1))n^3$ flops in the first iteration and $(6 + \beta \log_2(p-1))n^3$ flops in each remaining iteration.  Further savings in computational costs can be achieved when $p=2$; see~\cite[Section 4.2]{gawlik2018zolotarev} for details.

\subsection{Scalar iteration}

\begin{table}[t]
\centering
\begin{filecontents*}{convrate.dat}
   0.0000000e+00   5.0000000e-01             NaN   9.9999000e-01             NaN   9.0000000e-01             NaN
   1.0000000e+00   1.4863998e-01   1.1891198e+00   7.8214900e-01   7.8218811e-01   4.2646537e-02   8.9163315e-02
   2.0000000e+00   9.5360816e-03   2.9037756e+00   1.4268987e-02   4.8746750e-02   2.1115754e-11   8.2303854e-02
   3.0000000e+00   3.0324977e-06   3.4969594e+00   1.4379235e-11   2.4309165e-02   0.0000000e+00   0.0000000e+00
             NaN             NaN   3.5000000e+00             NaN   2.4305556e-02             NaN   8.2500000e-02
\end{filecontents*}
\pgfplotstabletypeset[
header=false,
font=\small,
clear infinite,
every head row/.style={before row=\midrule,after row=\midrule},
every last row/.style={after row=\midrule,before row=\midrule},
    every head row/.append style={
        before row=\toprule,
        before row/.add={}{%
        {} %
        & \multicolumn{4}{c|}{$(m,\ell,p)=(1,1,13)$} %
        & \multicolumn{4}{c|}{$(m,\ell,p)=(2,2,3)$} %
        & \multicolumn{4}{c|}{$(m,\ell,p)=(3,3,5)$}\\ \midrule
        }},
columns={0,1,2,3,4,5,6},
columns/0/.style={sci zerofill,column type/.add={|}{|},column name={$k$}},
columns/1/.style={dec sep align={c|},sci,sci 10e,sci zerofill,precision=4,column type/.add={}{|},column name={$\varepsilon_k$}},
columns/3/.style={dec sep align={c|},sci,sci 10e,sci zerofill,precision=4,column type/.add={}{|},column name={$\varepsilon_k$}}, 
columns/5/.style={dec sep align={c|},sci,sci 10e,sci zerofill,precision=4,column type/.add={}{|},column name={$\varepsilon_k$}}, 
columns/2/.style={dec sep align={c|},sci,sci 10e,sci zerofill,precision=2,column type/.add={}{|},column name={$\varepsilon_k/\varepsilon_{k-1}^{m+\ell+1}$}},
columns/4/.style={dec sep align={c|},sci,sci 10e,sci zerofill,precision=2,column type/.add={}{|},column name={$\varepsilon_k/\varepsilon_{k-1}^{m+\ell+1}$}},
columns/6/.style={dec sep align={c|},sci,sci 10e,sci zerofill,precision=2,column type/.add={}{|},column name={$\varepsilon_k/\varepsilon_{k-1}^{m+\ell+1}$}},
]
{convrate.dat}
\caption{Values of $\{\varepsilon_k\}_{k=1}^{3}$ generated by the iteration~(\ref{epsit}) with $f(z)=z^{1/p}$ for various choices of $m$, $\ell$, $p$, and $\varepsilon_0$. In each instance, the ratios $\varepsilon_k/\varepsilon_{k-1}^{m+\ell+1}$ approach the constant $C(m,\ell,p)$ given by~(\ref{constp}), whose value is recorded in the last row of the table for reference.}
\label{tab:convrate}
\end{table}

\paragraph{Asymptotic convergence rates}
To verify the asymptotic convergence rates predicted by Theorem~\ref{thm:pthroot}, we computed $\varepsilon_k = \frac{1-\alpha_k}{1+\alpha_k}$, $k=1,2,3$, for various choices of $m$, $\ell$, $p$, and $\varepsilon_0$.  
Table~\ref{tab:convrate} reports the results for three such choices.  (We selected values of $m$, $\ell$, $p$, and $\varepsilon_0$ so that the asymptotic regime was reached before convergence to machine precision occurred.)  The table demonstrates that the ratios $\varepsilon_k/\varepsilon_{k-1}^{m+\ell+1}$ approach the constant $C(m,\ell,p)$ given by~(\ref{constp}).  Note that the entry in the row $k=3$ of the last column should be ignored, since $\varepsilon_3$ is below machine precision in that instance.

\paragraph{Complex inputs}

\newcommand\scalefactor{0.25}
\begin{figure}\hspace{-0.2in}
\begin{tabular}{ccc} \vspace{0.05in}
\hspace{-0.0in}\includegraphics[scale=\scalefactor]{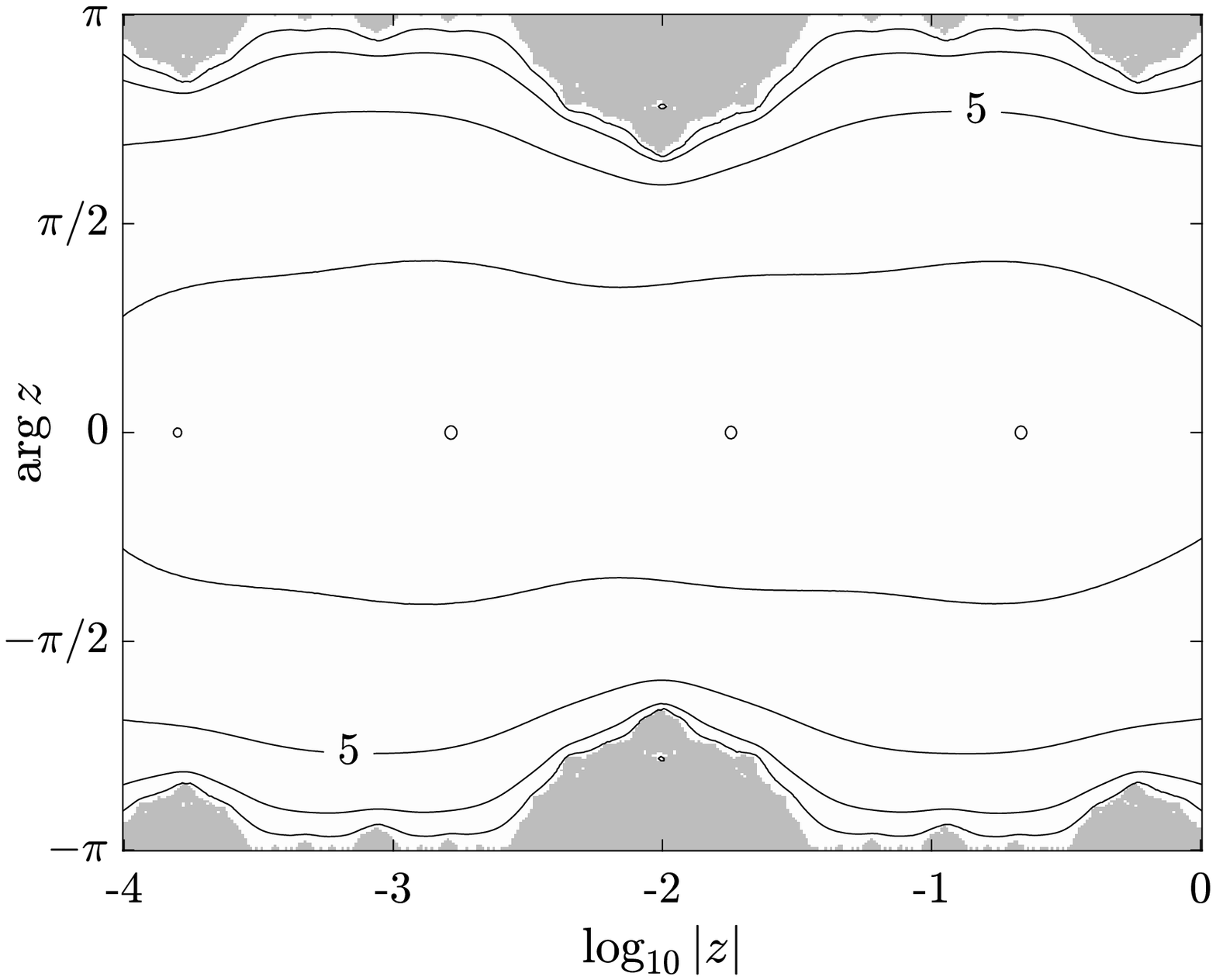} &
\hspace{-0.01in}\includegraphics[scale=\scalefactor]{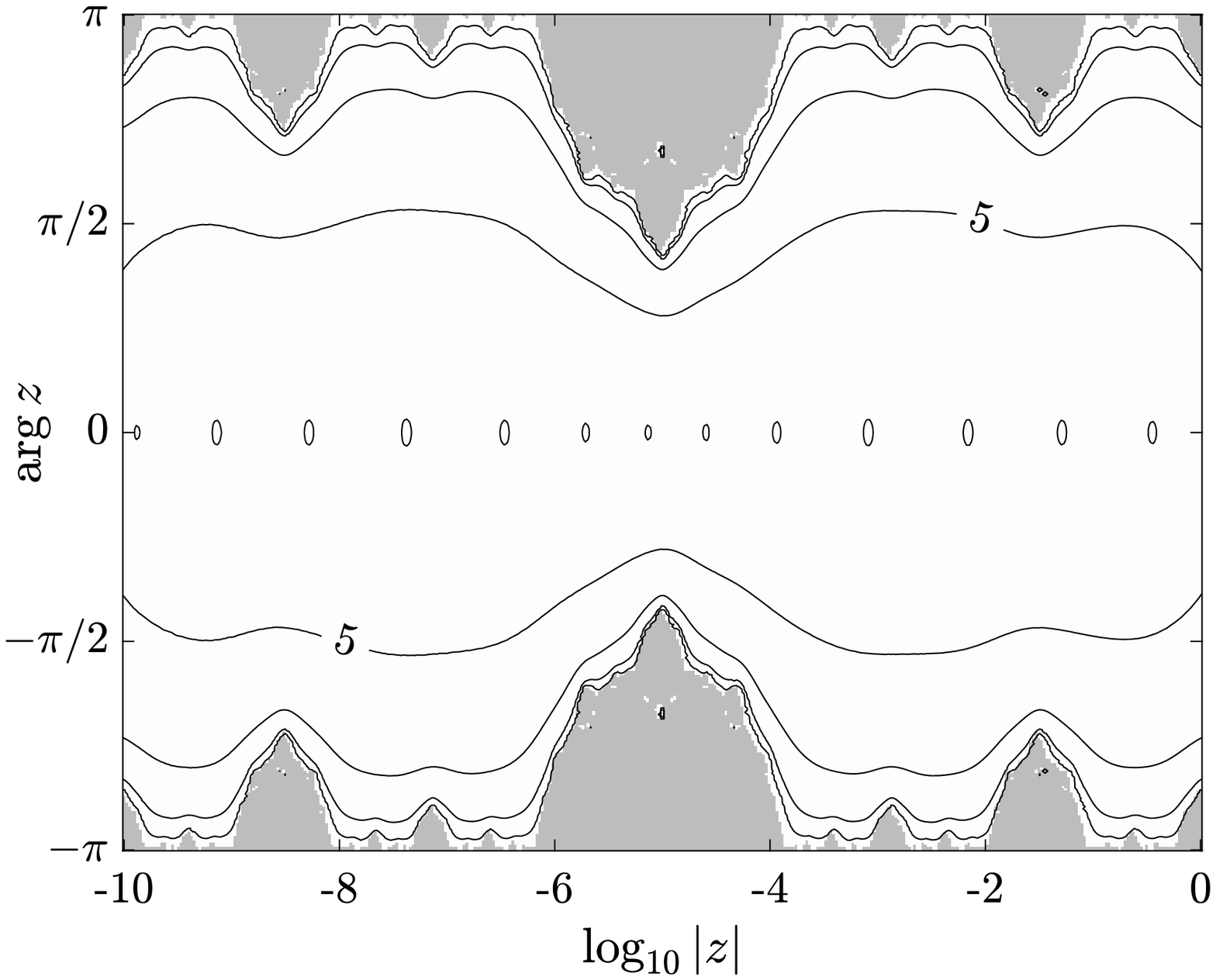} &
\hspace{-0.01in}\includegraphics[scale=\scalefactor]{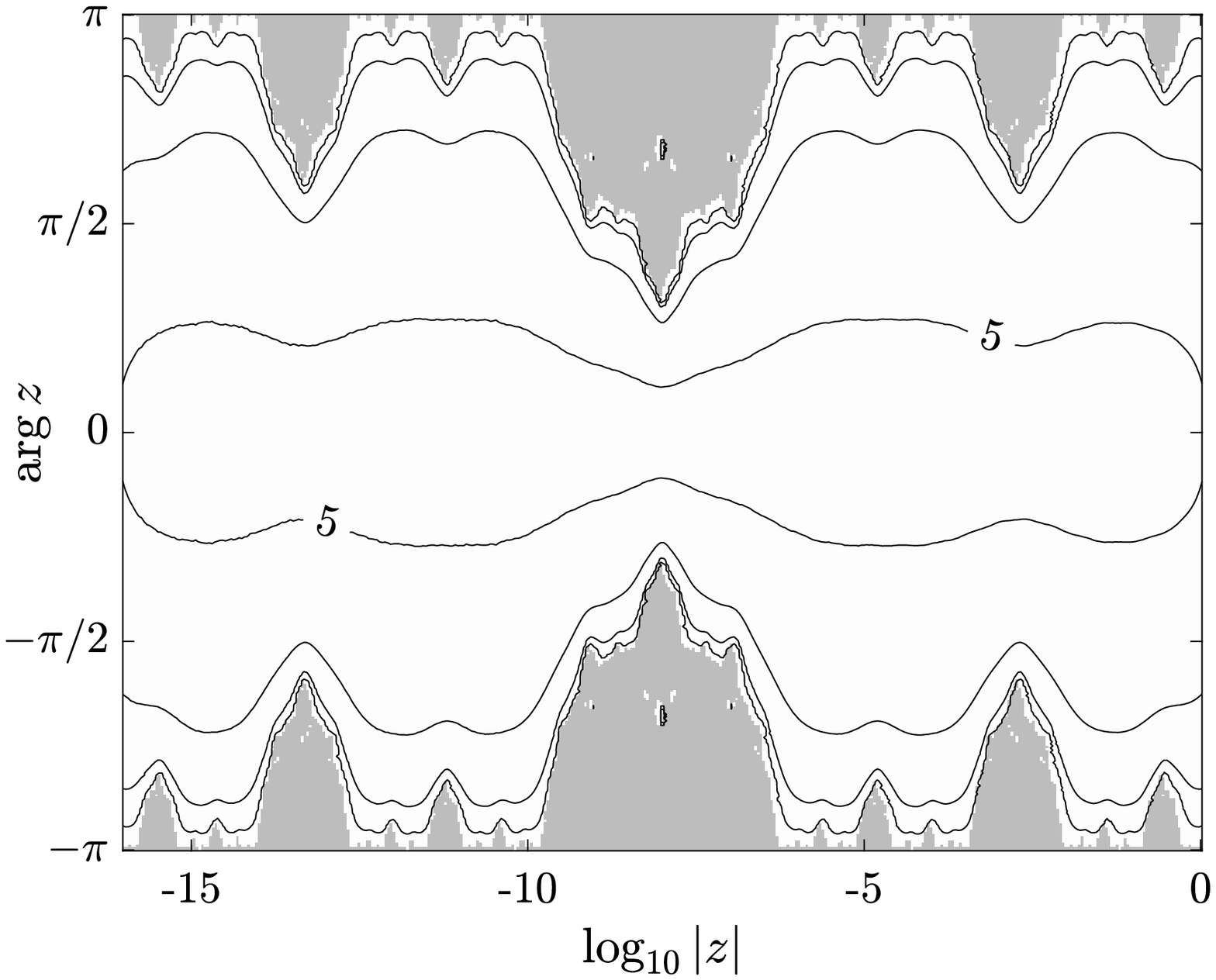} \\ \vspace{0.05in}
\hspace{-0.0in}\includegraphics[scale=\scalefactor]{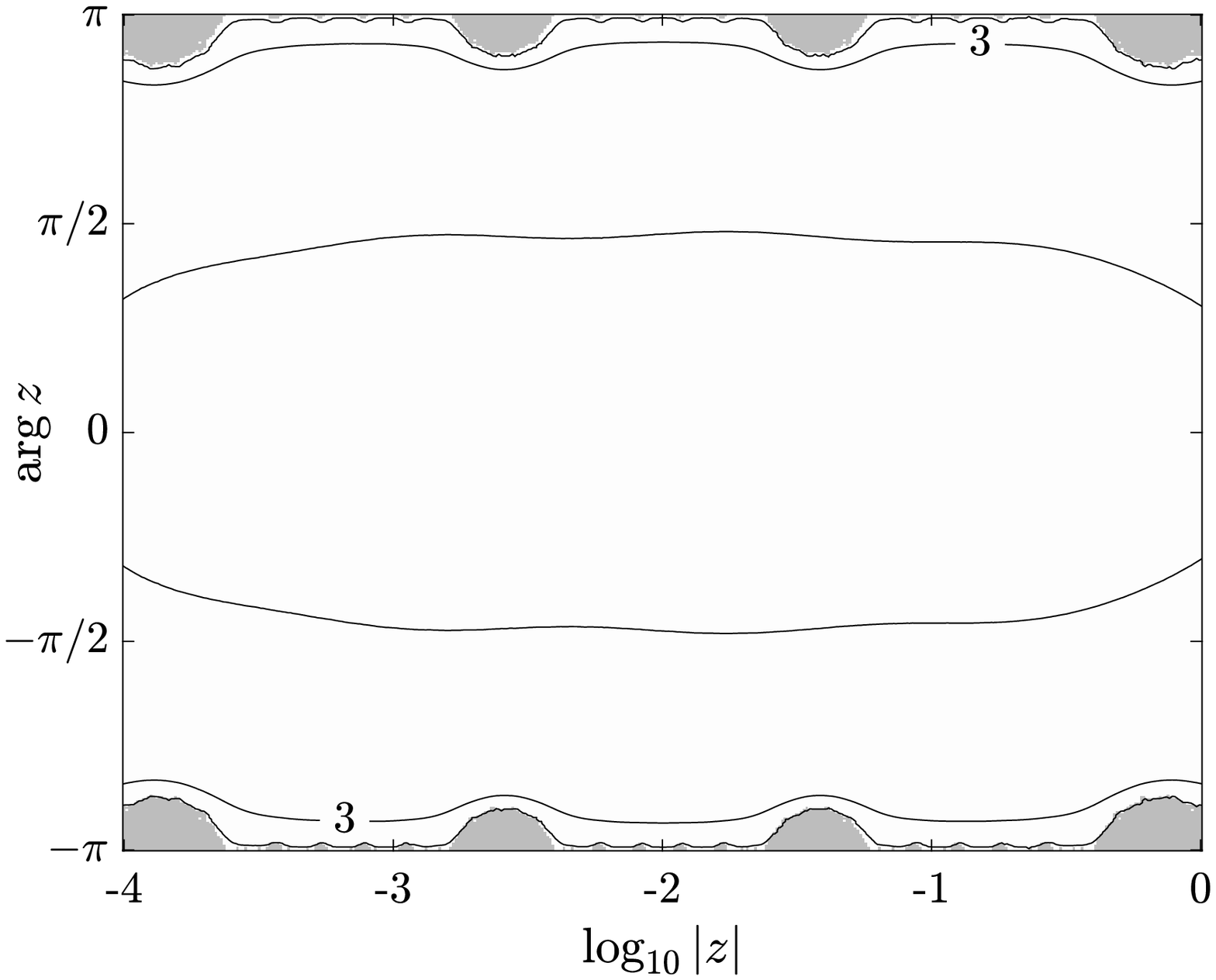} &
\hspace{-0.01in}\includegraphics[scale=\scalefactor]{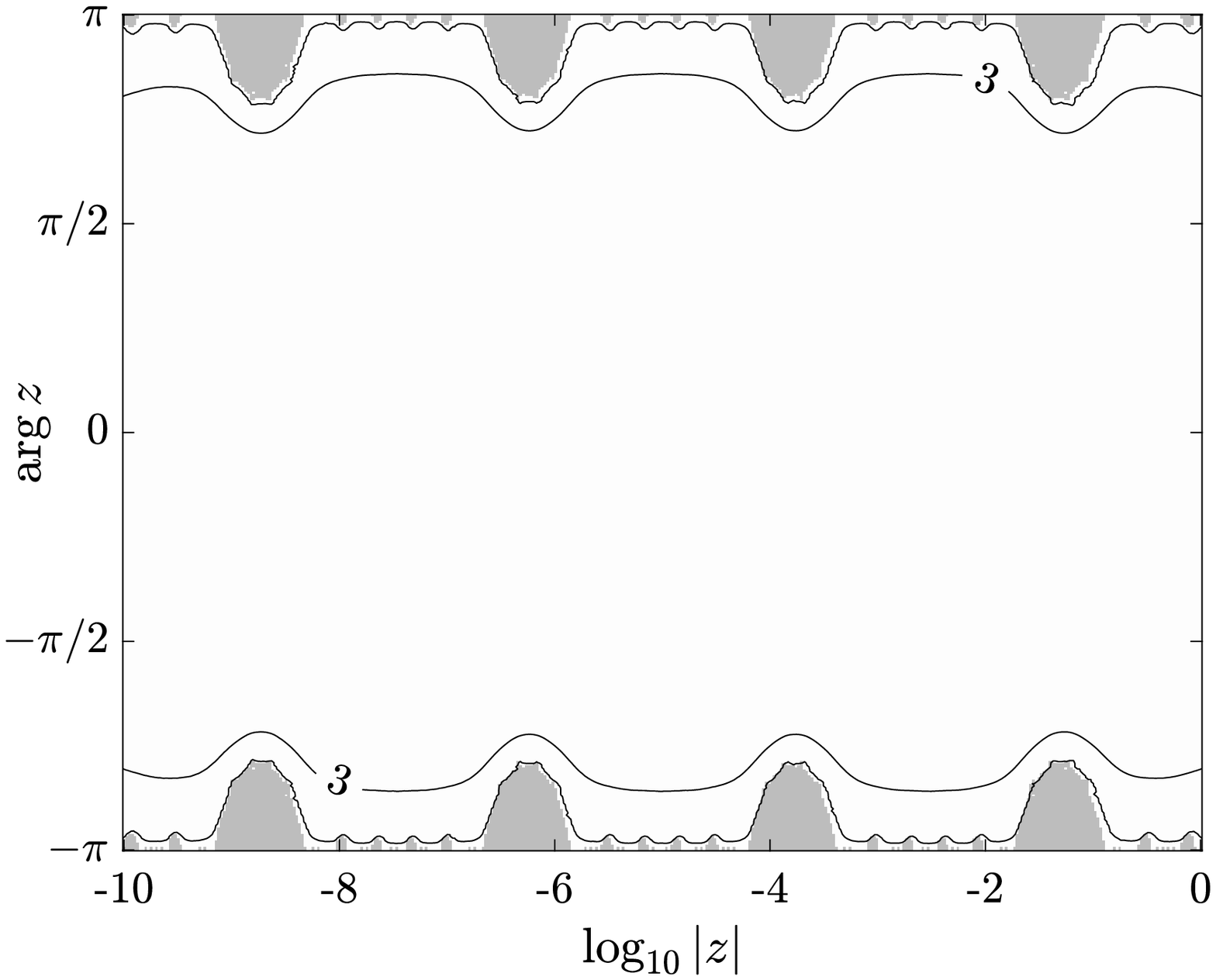} &
\hspace{-0.01in}\includegraphics[scale=\scalefactor]{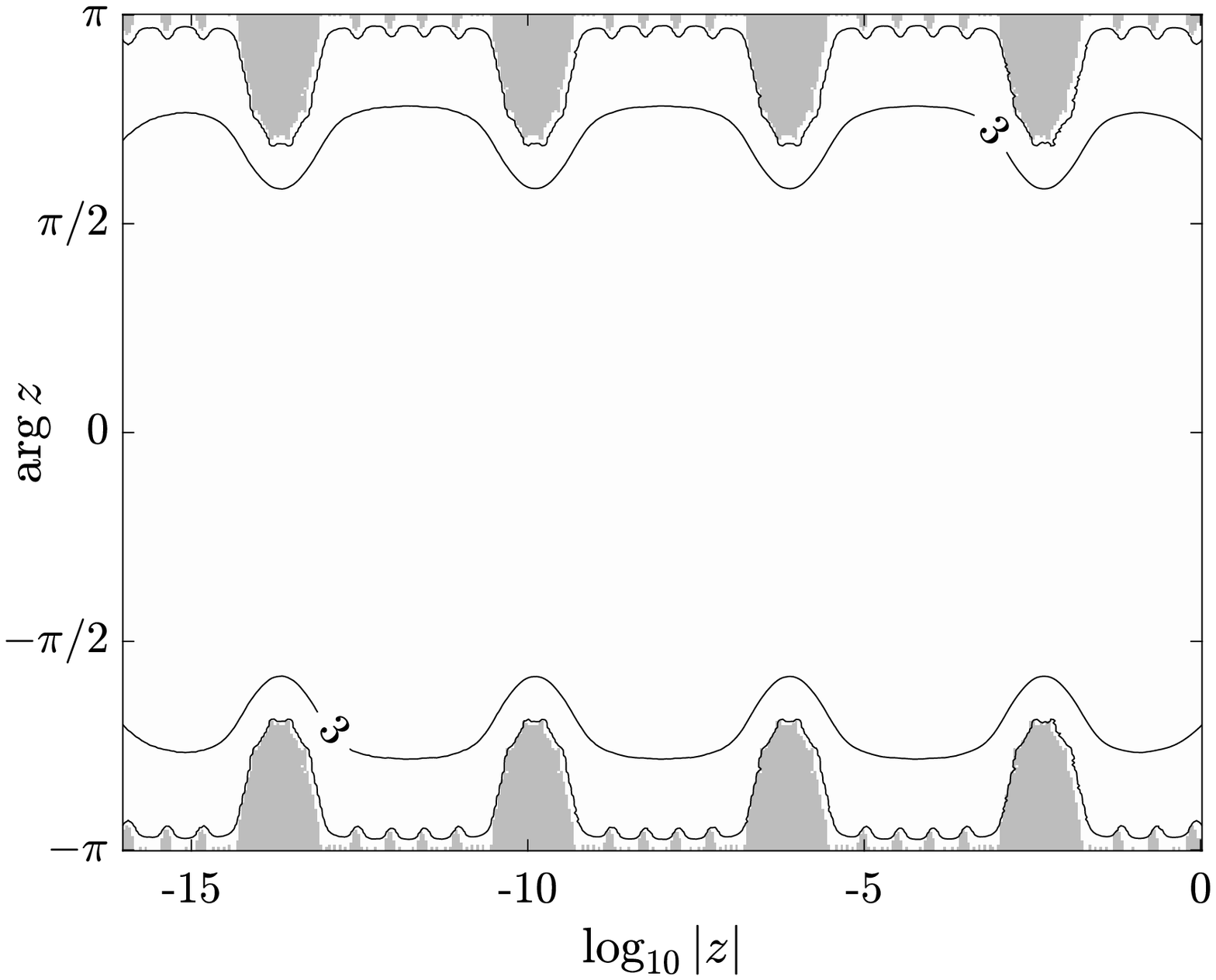} \\
\hspace{-0.0in}\includegraphics[scale=\scalefactor]{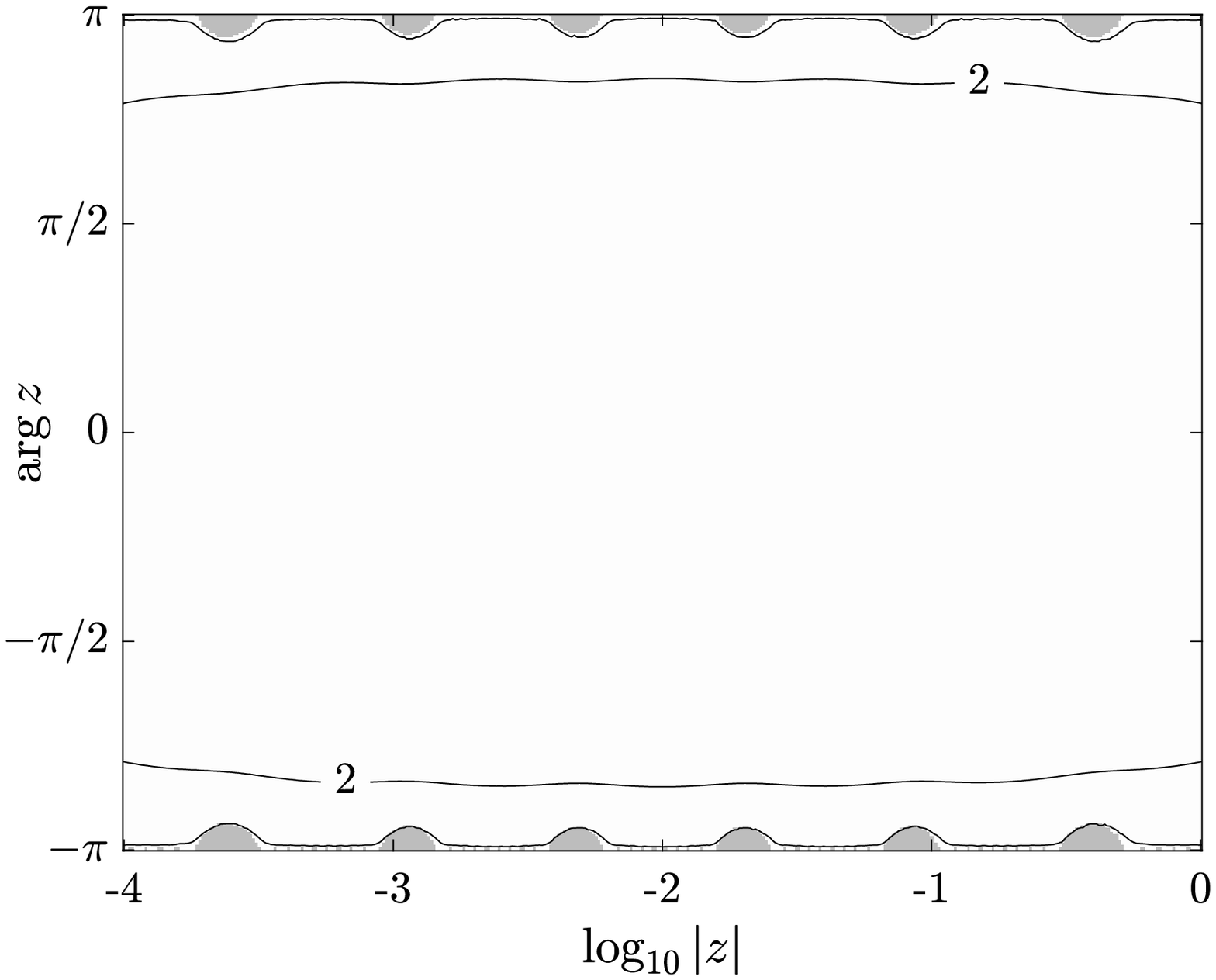} &
\hspace{-0.01in}\includegraphics[scale=\scalefactor]{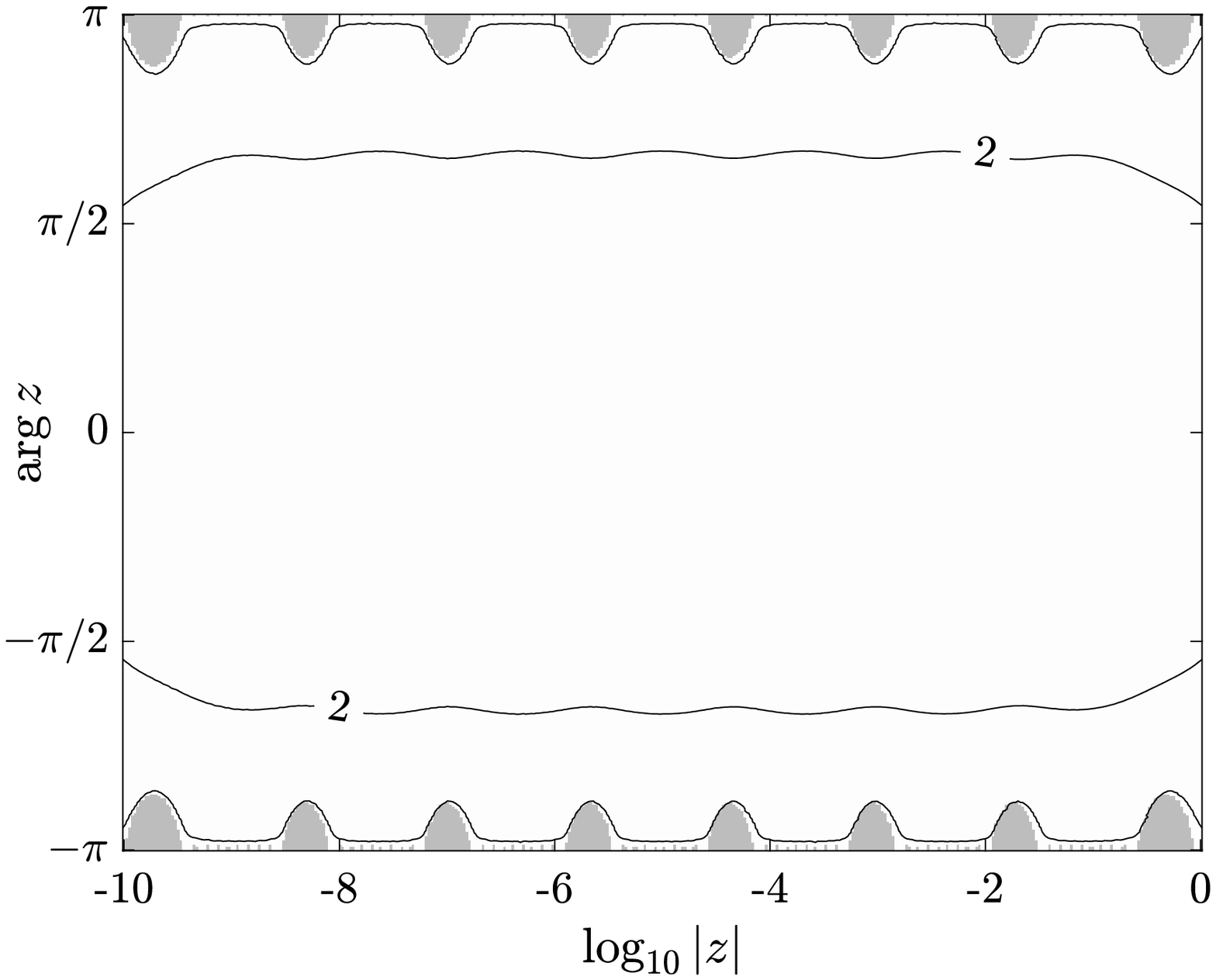} &
\hspace{-0.01in}\includegraphics[scale=\scalefactor]{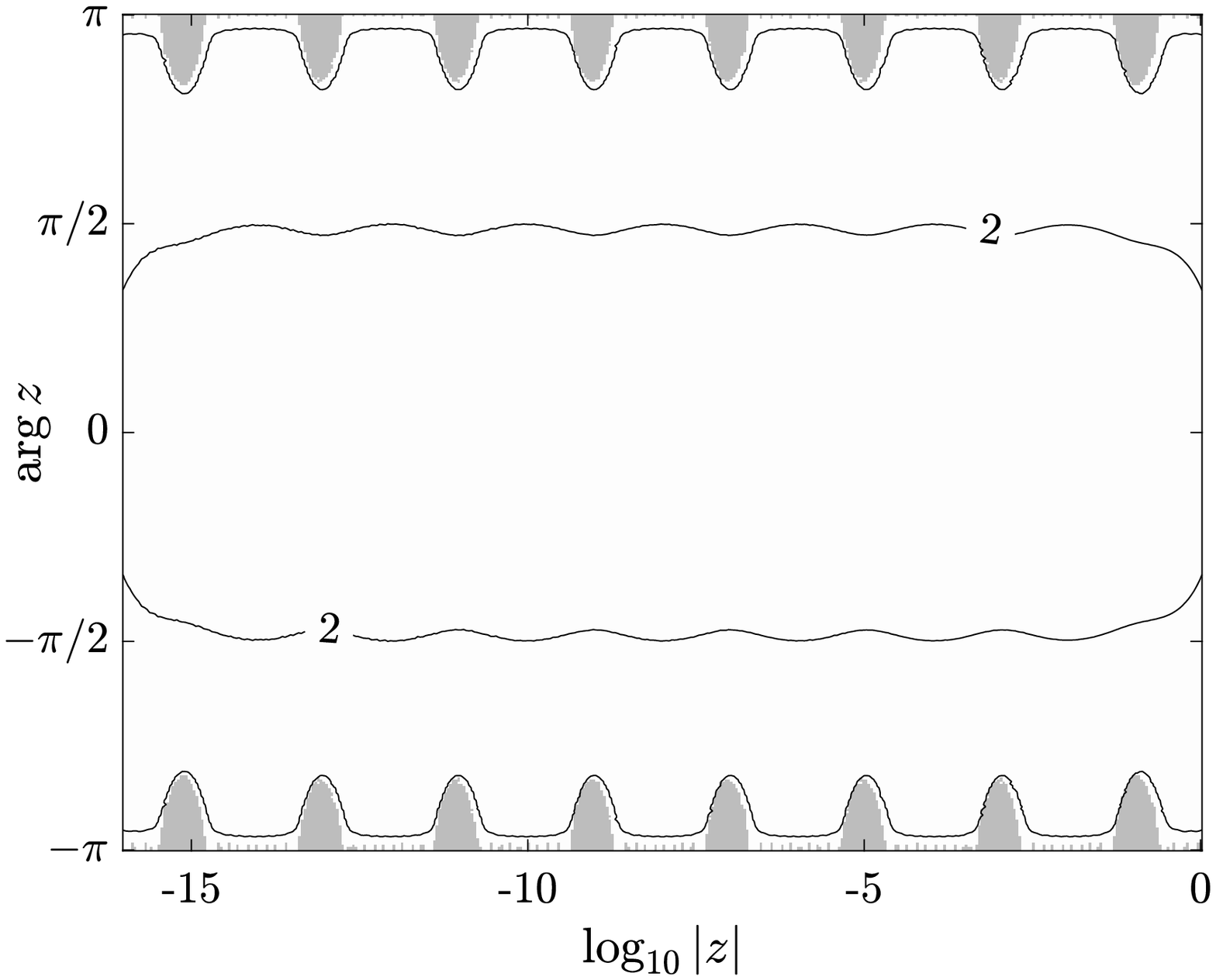} \\
\end{tabular}
\caption{Boundaries of the sets $\mathcal{S}(k;\delta,\alpha,m,\ell,p)$ with $\delta=10^{-14}$, $p=3$, $(m,\ell)=(1,1)$ (first row), $(m,\ell)=(4,4)$ (second row), $(m,\ell)=(8,8)$ (third row), $\alpha=10^{-4/3}$ (first column), $\alpha=10^{-10/3}$ (second column), and $\alpha=10^{-16/3}$ (third column).  In each plot, one of the boundaries has been selected arbitrarily and labelled with its index $k$.  Each unlabelled boundary has an index which differs by $+1$ from that of its nearest inner neighbor.  Shaded regions correspond to points $z$ for which $\lim_{k \rightarrow \infty} \widetilde{f}_k(z) \neq z^{1/p}$.}
\label{fig:mini}
\end{figure}

\begin{figure}\hspace{-0.2in}
\begin{tabular}{ccc} \vspace{0.05in}
\hspace{-0.0in}\includegraphics[scale=\scalefactor]{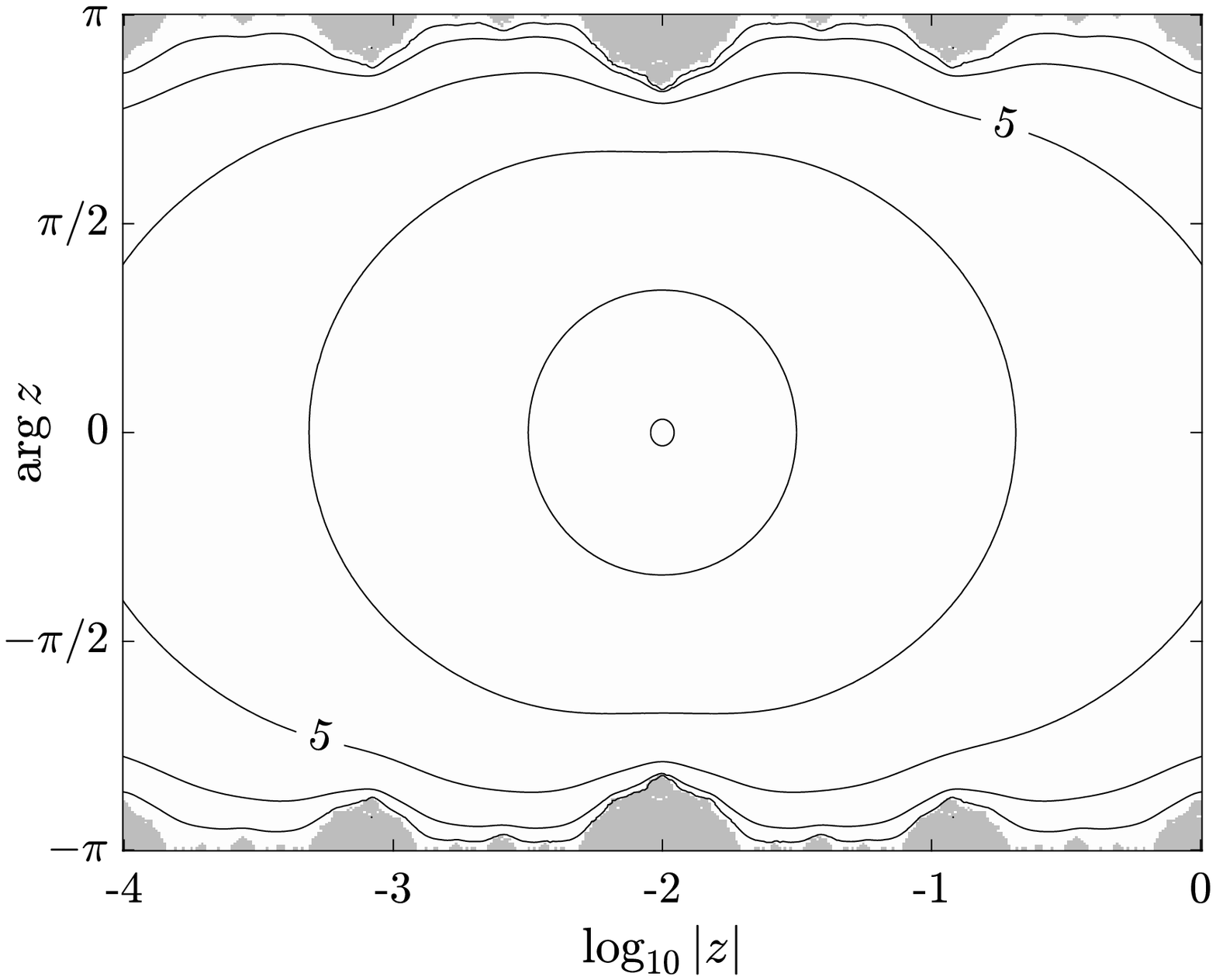} &
\hspace{-0.01in}\includegraphics[scale=\scalefactor]{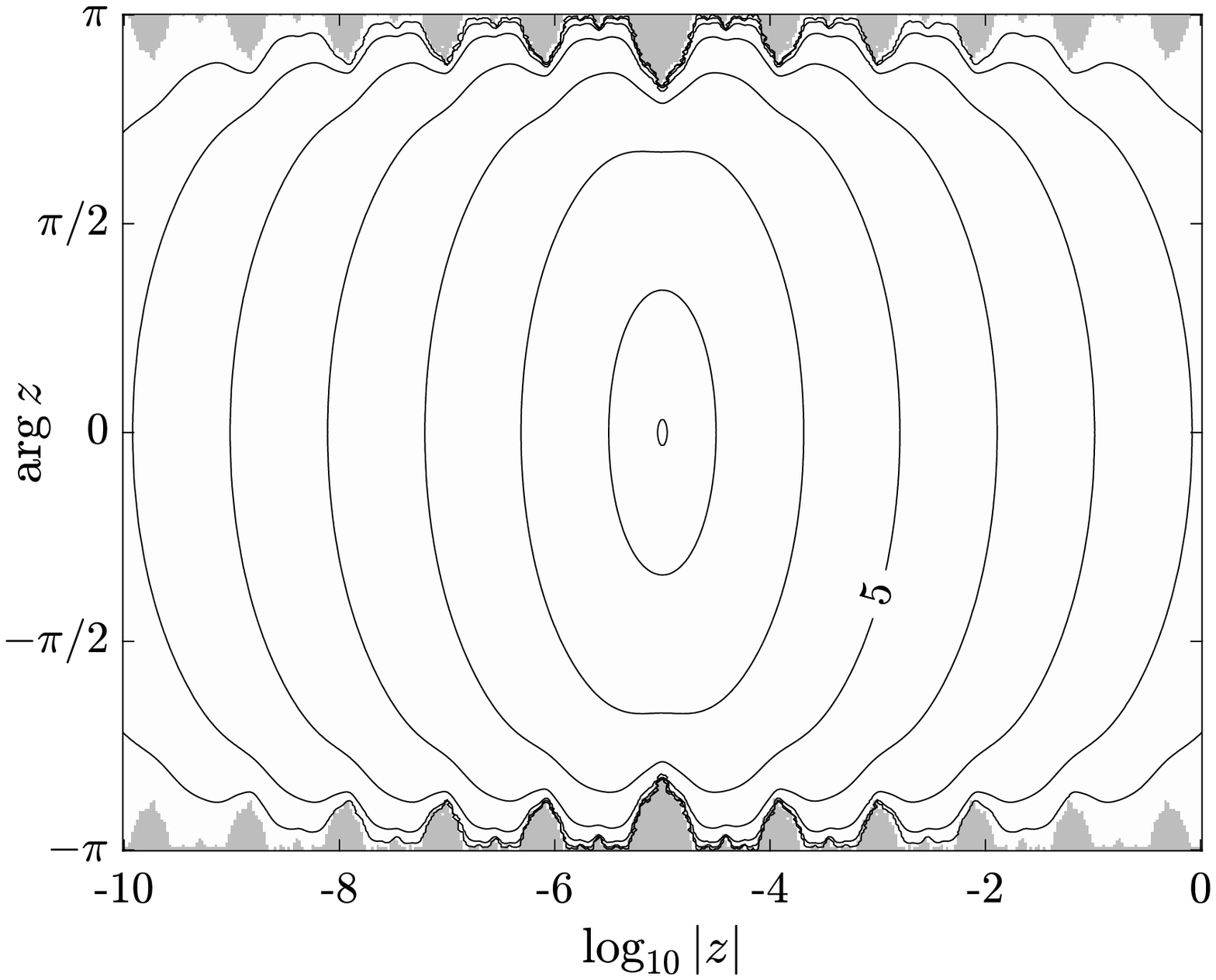} &
\hspace{-0.01in}\includegraphics[scale=\scalefactor]{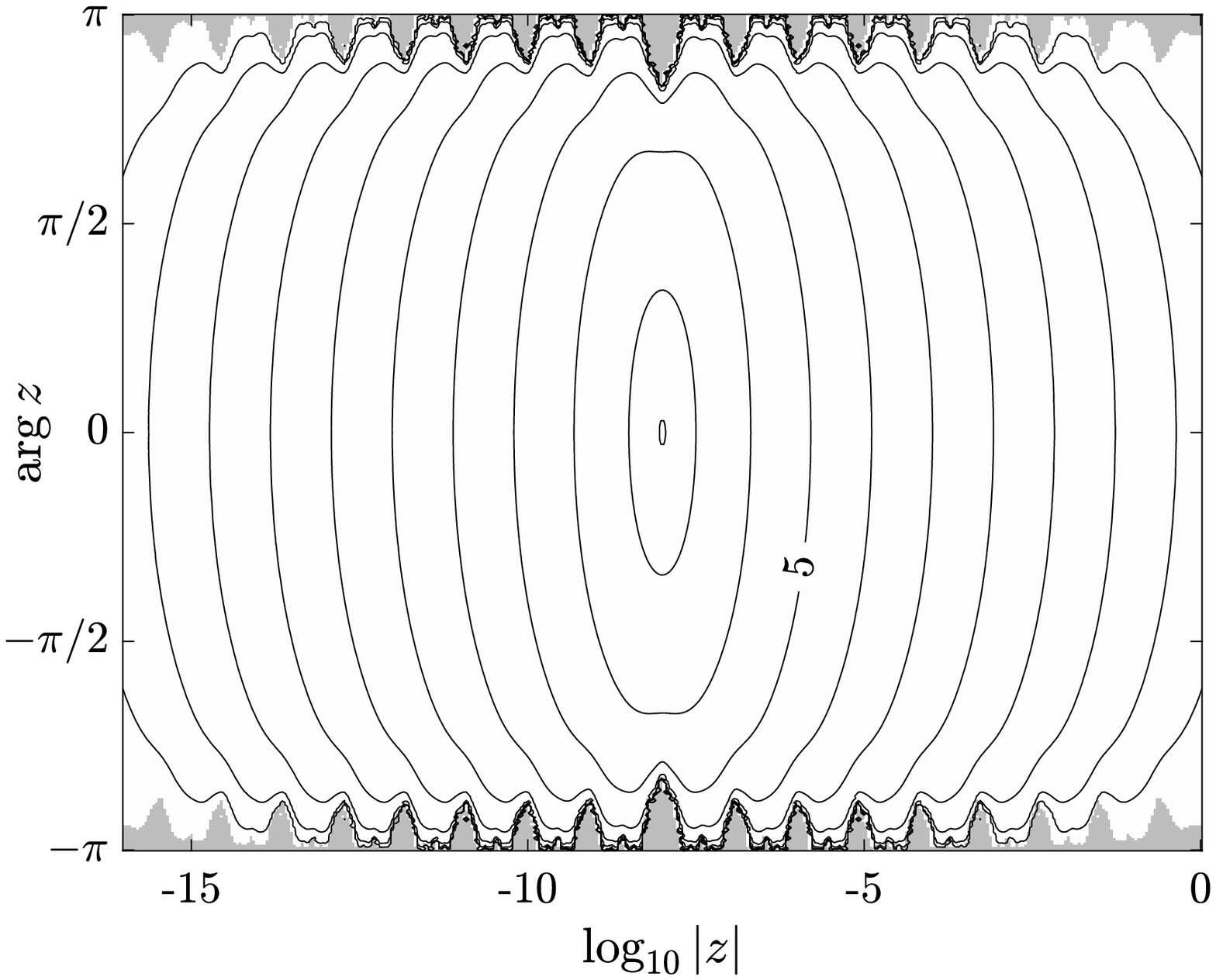} \\ \vspace{0.05in}
\hspace{-0.0in}\includegraphics[scale=\scalefactor]{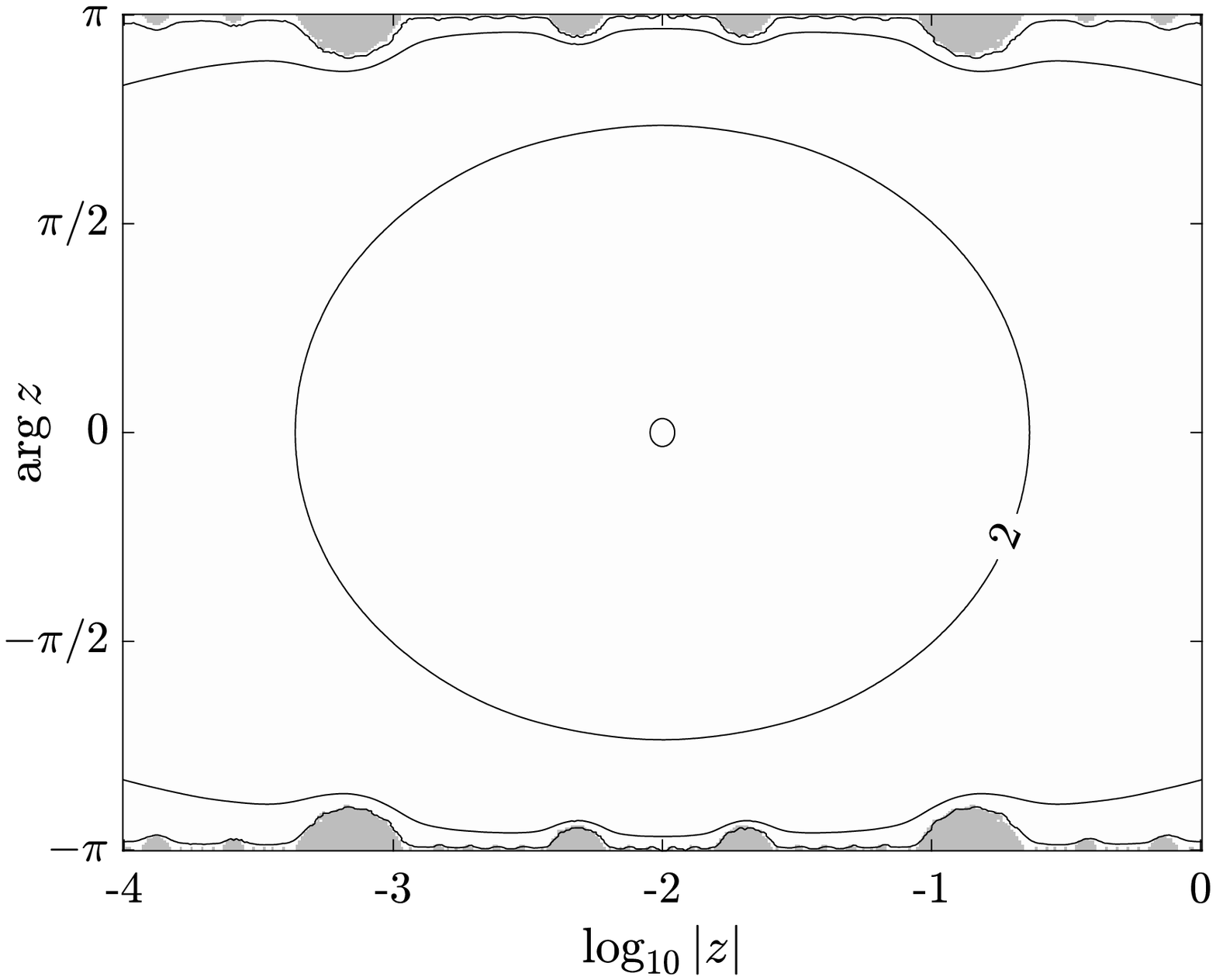} &
\hspace{-0.01in}\includegraphics[scale=\scalefactor]{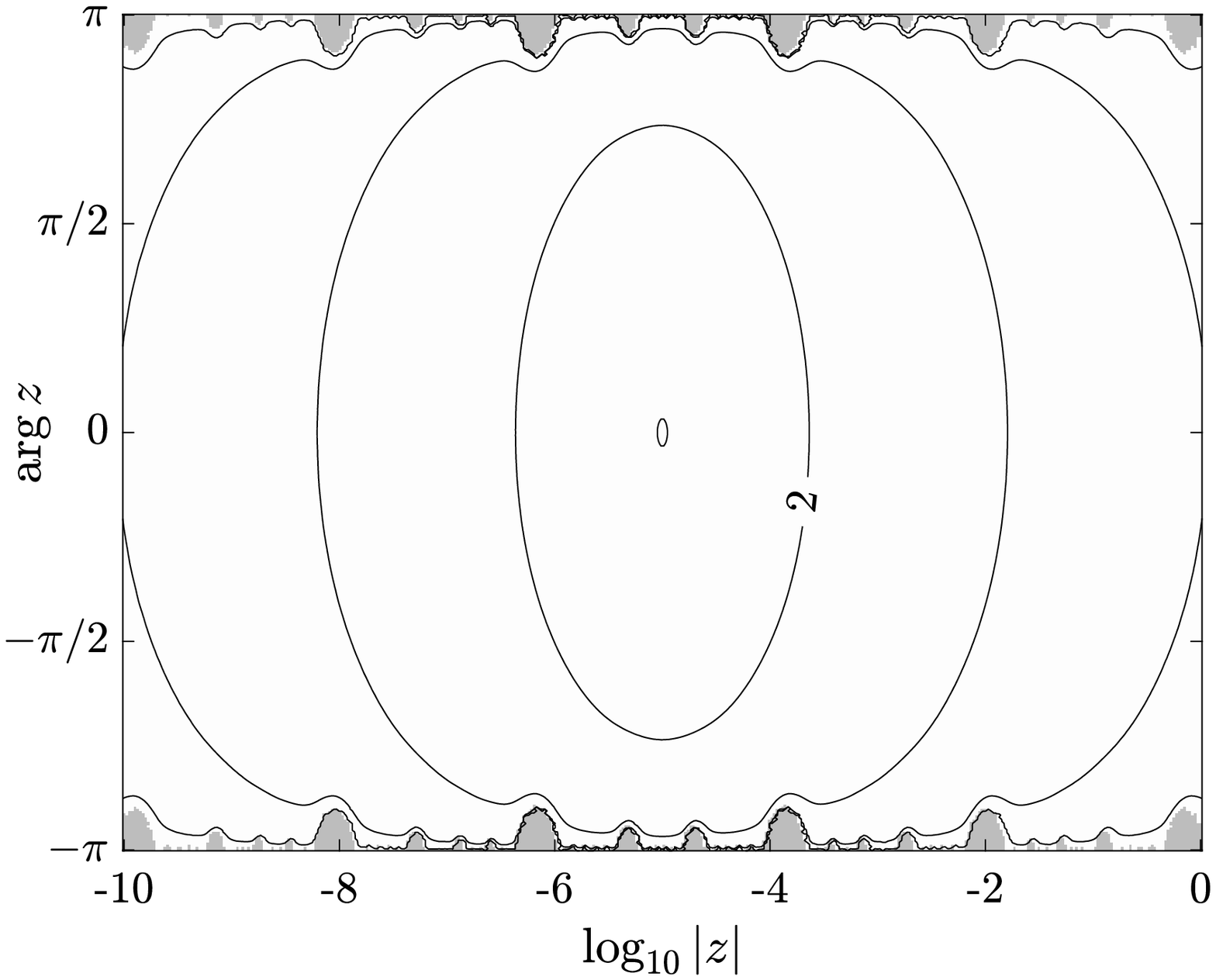} &
\hspace{-0.01in}\includegraphics[scale=\scalefactor]{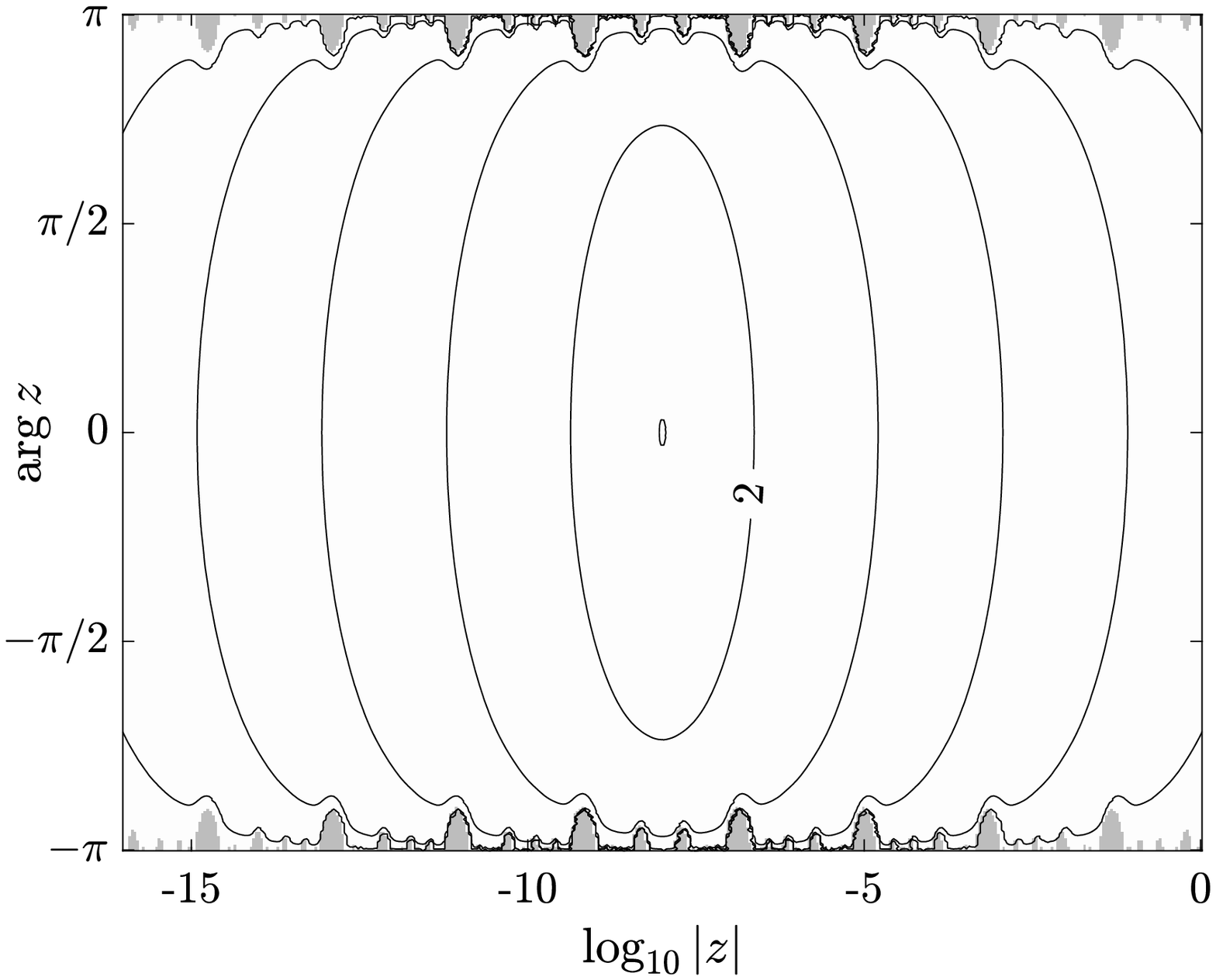} \\
\hspace{-0.0in}\includegraphics[scale=\scalefactor]{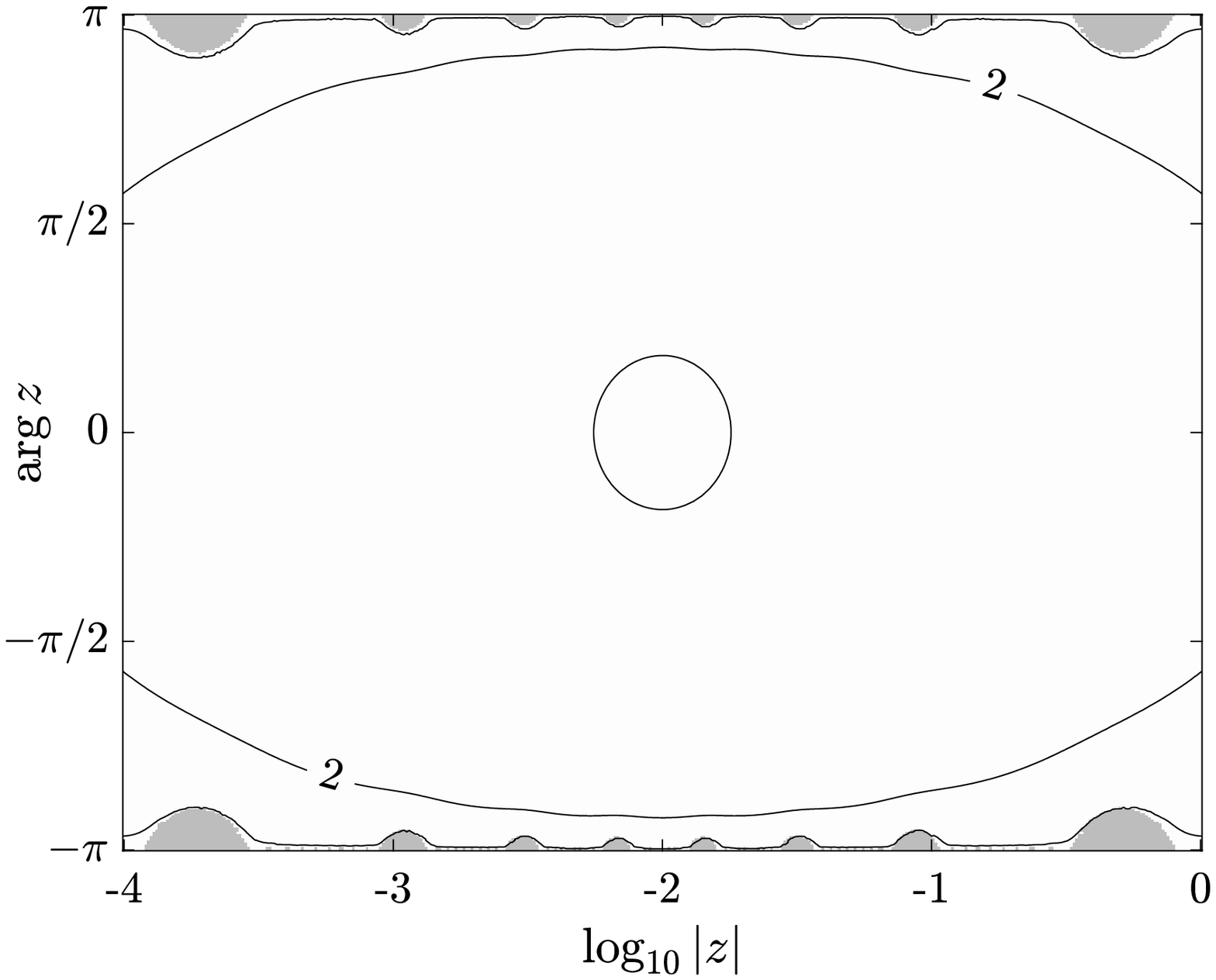} &
\hspace{-0.01in}\includegraphics[scale=\scalefactor]{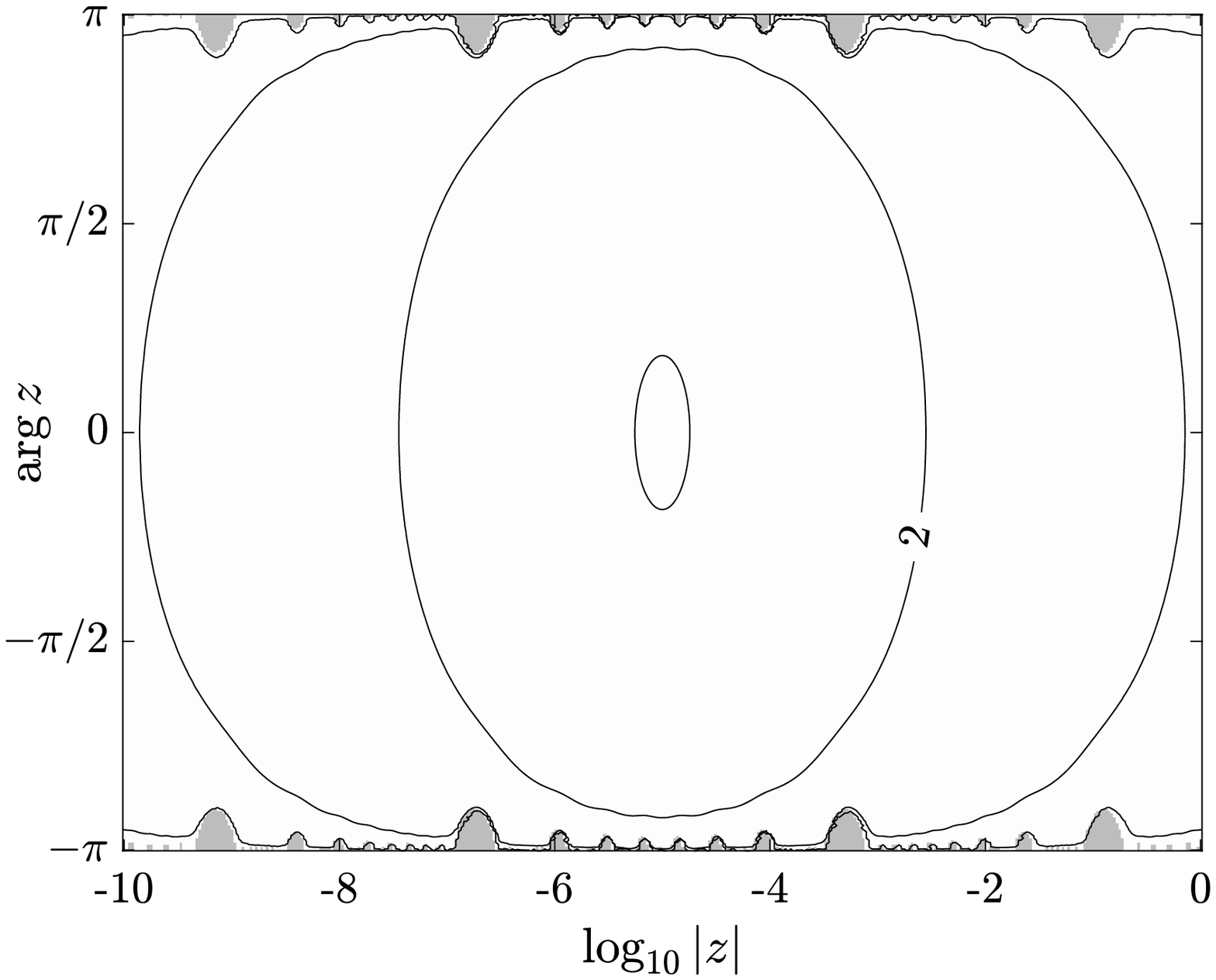} &
\hspace{-0.01in}\includegraphics[scale=\scalefactor]{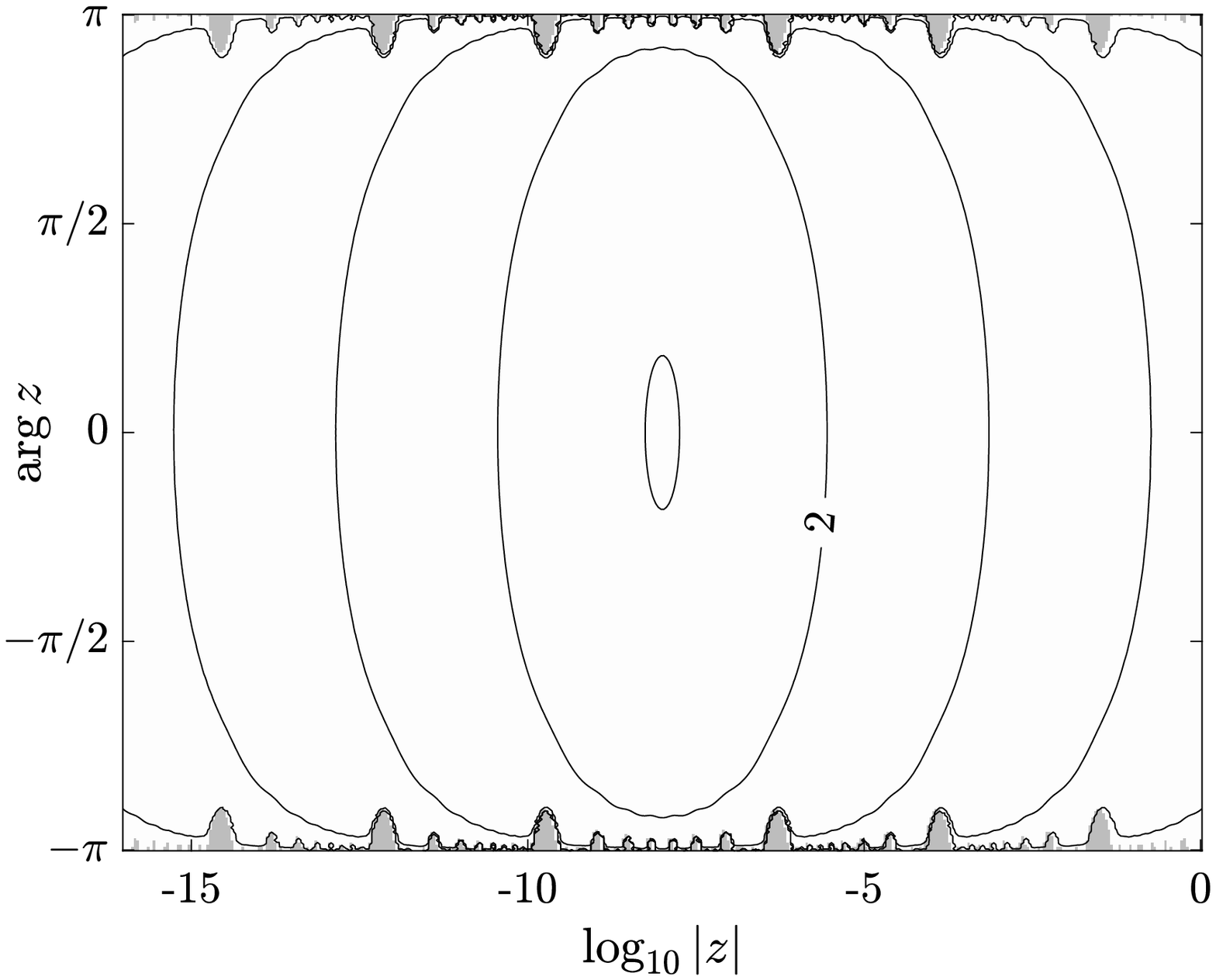} \\
\end{tabular}
\caption{Boundaries of the sets $\mathcal{T}(k;\delta,\alpha,m,\ell,p)$ with the same parameters as in Fig.~\ref{fig:mini}.}
\label{fig:pade}
\end{figure}

To study the behavior of the rational function $\widetilde{f}_k(z)$ generated by the type-$(m,\ell)$ iteration~(\ref{pthrootit1}-\ref{pthrootit2}), we numerically computed the sets
\[
\mathcal{S}(k) = \mathcal{S}(k;\delta,\alpha,m,\ell,p) =  \left\{z \in \mathbb{C} \colon \left| \frac{\widetilde{f}_k(z)-z^{1/p}}{z^{1/p}} \right| \le \delta \right\}
\]
for various choices of $\delta$, $\alpha$, $m$, $\ell$, and $p$.  The boundaries of these sets are plotted in Fig.~\ref{fig:mini}.  They are plotted in the $(\log_{10}|z|,\arg z)$ coordinate plane rather than the usual $(\Re z, \Im z)$ coordinate plane to facilitate viewing.  The shaded regions in the plots correspond to points $z \in \mathbb{C}$ for which $\lim_{k\rightarrow \infty} \widetilde{f}_k(z) \neq z^{1/p}$.  Numerical evidence indicates that at these points, $\lim_{k\rightarrow \infty} \widetilde{f}_k(z) \in  \{e^{2\pi i j/p} z^{1/p} \mid j \in \{1,2,\dots,p-1\} \}$.  Furthermore, the shaded regions have a fractal structure.  Both of these phenomena are typical features of iterations for the $p^{th}$ root when $p>2$~\cite{cardoso2011iteration}.

Fig.~\ref{fig:mini} gives valuable insight into the behavior of the matrix iteration~(\ref{pthrootmit1}-\ref{pthrootmit2}) (and, of course, its coupled counterpart~(\ref{pcoupled1}-\ref{pcoupled3})).  Indeed, if $A$ is a normal matrix with eigenvalues in $\mathcal{S}(k)$, then the iteration~(\ref{pthrootmit1}-\ref{pthrootmit2}) converges in at most $k$ iterations with a relative tolerance $\delta$ in the $2$-norm.  As an example, the plot in row 3, column 2 of Fig.~\ref{fig:mini} demonstrates that $\mathcal{S}(2)$ contains the set 
\[
\{z \in \mathbb{C} \mid \log_{10} |z| \in [-10,0],\, \arg z \in [-\pi/2,\pi/2] \} 
\]
when $(m,\ell) = (8,8)$, $p=3$, and $\alpha=10^{-10/3}$.  It follows that the type-$(8,8)$ iteration~(\ref{pthrootmit1}-\ref{pthrootmit2}) converges to $A^{1/3}$ in at most $2$ iterations for any normal matrix $A$ with spectrum in the right half plane and $|\lambda_{\mathrm{max}}(A)/\lambda_{\mathrm{min}}(A)| \le 10^{10}$.

For comparison, Fig.~\ref{fig:pade} shows the boundaries of the sets
\[
\mathcal{T}(k) = \mathcal{T}(k;\delta,\alpha,m,\ell,p) =  \left\{z \in \mathbb{C} \colon \left| \frac{\widetilde{f}_k(z/\alpha^{p/2})-(z/\alpha^{p/2})^{1/p}}{(z/\alpha^{p/2})^{1/p}} \right| \le \delta \right\},
\]
where this time $\widetilde{f}_k(z)$ is the rational function generated by~(\ref{pthrootit1}-\ref{pthrootit2}) with the initial condition $\alpha_0 = \alpha$ replaced by $\alpha_0=1$.  By Proposition~\ref{prop:pade}, the sets $\mathcal{T}(k)$ characterize the convergence behavior of the Pad\'e iteration~(\ref{padep}) (and its coupled counterpart~(\ref{padecoupled1}-\ref{padecoupled2})) with the initial iterate scaled by $1/\alpha^{p/2}$.  

Notice that for small $\alpha$ (the two rightmost columns of Fig.~\ref{fig:pade}), the sets $\mathcal{T}(k)$ do not contain scalars with extreme magnitudes ($|z|=\alpha^p$ and $|z|=1$) unless $k$ is relatively large.  Comparing, for instance, the bottom right plots in Figs.~\ref{fig:mini} and~\ref{fig:pade}, we see that if $A$ is Hermitian positive definite with spectrum in $[10^{-16},1]$, then the type-$(8,8)$ rational minimax iteration~(\ref{pthrootit1}-\ref{pthrootit2}) converges in at most $2$ iterations, whereas the type-$(8,8)$ Pad\'e iteration~(\ref{padep}) converges in at most $5$.  The same observation holds, in fact, for the type-$(6,6)$ and type-$(7,7)$ iterations, which are not shown in Figs.~\ref{fig:mini}-\ref{fig:pade}.  This is entirely analagous to the behavior observed in the case $p=2$ in~\cite[Section 5.1]{gawlik2018zolotarev}.  In fact, with the exception of the low-order iterations, Figs.~\ref{fig:mini}-\ref{fig:pade} bear a rather strong resemblance to Figs. 1-2 of~\cite{gawlik2018zolotarev}.

It is worth noting that for the low-order iterations, the sets $\{z \in \mathbb{C} \mid \lim_{k\rightarrow \infty} \widetilde{f}_k(z) \neq z^{1/p}\}$ occupy more of the complex plane when $\widetilde{f}_k(z)$ is generated from the rational minimax iteration than when $\widetilde{f}_k(z)$ is generated from the Pad\'e iteration (see the shaded regions in row 1 of Figs.~\ref{fig:mini}-\ref{fig:pade}).  This appears to be a drawback of the low-order rational minimax iterations.  The moderate-order and high-order iterations do not suffer as much from this issue; compare the shaded regions in the bottom two rows of Figs.~\ref{fig:mini}-\ref{fig:pade}, which occupy only a small neighborhood of the nonpositive real axis ($|\arg z|=\pi$).  The latter observation suggests that for moderate to high $m$ and $\ell$, it is safe to apply Algorithm~\ref{alg:minimax} to matrices with spectrum contained in $\{z \in \mathbb{C} \colon |\arg z| \le \Theta\}$, where $\Theta<\pi$ is close to $\pi$.  For matrices with eigenvalues that lie very near but not on the nonpositive real axis, a simple workaround is to compute $A^{1/2}$ using any algorithm for the matrix square root, and then compute $((A^{1/2})^{1/p})^2$.  One can also compute $((A^{1/2^s})^{1/p})^{2^s}$ with $s >1$, as in~\cite{guo2006schur,higham2011schur}, but the advantages of minimax approximation over Pad\'e approximation become less pronounced as $s$ increases, since $A^{1/2^s}$ has eigenvalues clustered near $1$ for large $s$.

\subsection{Matrix iteration} 

\SaveVerb{funm}|funm|

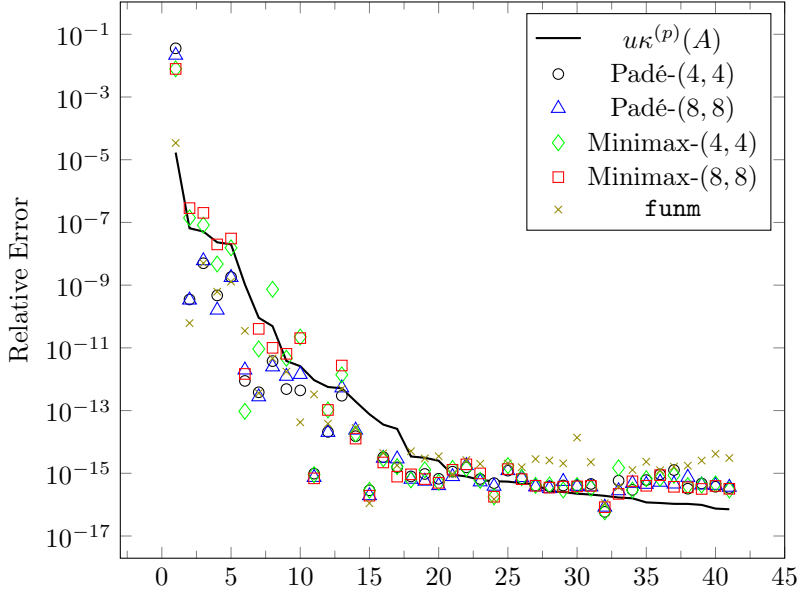
\begin{figure}
\centering
\pgfplotstableread{err.dat}{\tests}
\begin{tikzpicture} 
\begin{semilogyaxis}[width=0.8\textwidth,
minor tick num=1,
ylabel=Relative Error]
\addplot [black,thick] table [x=0, y=1] {\tests};
\addplot+[only marks,mark=*,black,mark options={fill=white}] table [x=0, y=2] {\tests};
\addplot+[only marks,mark=triangle,blue,mark options={scale=1.5}] table [x=0, y=3] {\tests};
\addplot+[only marks,mark=diamond,green,mark options={scale=1.5}] table [x=0, y=4] {\tests};
\addplot+[only marks,mark=square,red] table [x=0, y=5] {\tests};
\addplot+[only marks,mark=x,olive] table [x=0, y=6] {\tests};
\legend{$u\kappa^{(p)}(A)$\\Pad\'e-$(4,4)$\\Pad\'e-$(8,8)$\\Minimax-$(4,4)$\\Minimax-$(8,8)$\\\protect\UseVerb{funm}\\};
\end{semilogyaxis}
\end{tikzpicture}
\caption{Relative errors committed by the Pad\'e iterations of type $(4,4)$ and $(8,8)$, the minimax iterations of type $(4,4)$ and $(8,8)$, and the Matlab function \textup{\protect\UseVerb{funm}}.  Results are shown for $41$ tests with $p=3$, ordered by decreasing condition number $\kappa^{(p)}(A)$.}
\label{fig:err}
\end{figure}

\begin{table}
\centering
\begin{filecontents*}{iter.dat}
   0.0000000e+00   1.7000000e+01   1.2000000e+01   6.0000000e+00   4.0000000e+00   2.0000000e+00
   0.0000000e+00   2.7000000e+01   7.0000000e+00   6.0000000e+00   1.0000000e+00   0.0000000e+00
   0.0000000e+00   1.7000000e+01   2.0000000e+01   2.0000000e+00   1.0000000e+00   1.0000000e+00
   0.0000000e+00   3.4000000e+01   6.0000000e+00   1.0000000e+00   0.0000000e+00   0.0000000e+00
\end{filecontents*}
\pgfplotstabletypeset[
every head row/.style={before row=\toprule,after row=\midrule},
every last row/.style={after row=\bottomrule},
columns={leftcol,0,1,2,3,4,5},
create on use/leftcol/.style={create col/set list={{Pad\'e-$(4,4)$},{Pad\'e-$(8,8)$},{Minimax-$(4,4)$},{Minimax-$(8,8)$}}},
columns/leftcol/.style={string type,column type/.add={|}{|},column name={Iterations}},
columns/0/.style={fixed,precision=1,verbatim,column type/.add={}{|},column name={$1$}},
columns/1/.style={fixed,precision=1,verbatim,column type/.add={}{|},column name={$2$}},
columns/2/.style={fixed,precision=1,verbatim,column type/.add={}{|},column name={$3$}},
columns/3/.style={fixed,precision=1,verbatim,column type/.add={}{|},column name={$4$}},
columns/4/.style={fixed,precision=1,verbatim,column type/.add={}{|},column name={$5$}},
columns/5/.style={fixed,verbatim,column type/.add={}{|},column name={$\ge 6$}},
]
{iter.dat}
\vspace{0.05in}
\caption{Number of iterations used by each iterative method in the tests appearing in Fig.~\ref{fig:err}.}
\label{tab:iter}
\end{table}

To test Algorithm~\ref{alg:minimax}, we applied it to a collection of matrices of size $10 \times 10$ from the Matrix Computation Toolbox~\cite{Higham:MCT}.  We selected those $10 \times 10$ matrices in the toolbox with condition number $\le u^{-1}$ (where $u = 2^{-53}$ denotes the unit roundoff) and with spectrum contained in the sector $\{z \in \mathbb{C} \colon |\arg z| < 0.9\pi \}$.  We also included those matrices whose spectrum could be rotated into the aforementioned sector by multiplying $A$ by a suitable scalar $e^{i \theta}$, $\theta \in [0,2\pi]$.  A total of 41 matrices met these criteria.  

Fig.~\ref{fig:err} plots the relative error $\|\widehat{X}-A^{1/p}\|_\infty / \|A^{1/p}\|_\infty$ in the computed $p^{th}$ root $\widehat{X}$ of $A$ for each of the $41$ matrices, where $p=3$.  The tests are sorted in order of decreasing $\kappa^{(p)}(A)$, where 
\[
\kappa^{(p)}(A) = \frac{\|A\|_F}{\|X\|_F} \left\| \left(\sum_{j=1}^p (X^{p-j})^T \otimes X^{j-1} \right)^{-1} \right\|_2
\]
denotes the Frobenius-norm relative condition number of the matrix $p^{th}$ root $X$ of $A$~\cite[Problem 7.4]{higham2008functions}.  Results for five methods are shown: the rational minimax iterations~(\ref{pcoupled1}-\ref{pcoupled3}) of type $(4,4)$ and $(8,8)$, the Pad\'e iterations~(\ref{padecoupled1}-\ref{padecoupled2}) of type $(4,4)$ and $(8,8)$, and the built-in Matlab function \verb$funm$.  The Pad\'e iterations were implemented using Algorithm~\ref{alg:minimax} with Lines~\ref{line:scaling}-\ref{line:alpha0} replaced by $\tau = 1/\sqrt{|\lambda_{\mathrm{min}}(A)\lambda_{\mathrm{max}}(A)|}$ and $\alpha_0=1$.  The results indicate that the algorithms under consideration behave in a forward stable way, with relative errors mostly lying within a small factor of $u \kappa^{(p)}(A)$.

In Table~\ref{tab:iter}, the number of iterations used by each iterative method on the 41 tests are recorded.  In analogy with the results of~\cite{gawlik2018zolotarev}, the rational minimax iterations very often converged more quickly than the Pad\'e iterations on these tests.

\section{Conclusion}

This paper has constructed and analyzed a family of iterations for computing the matrix $p^{th}$ root using rational minimax approximants of the function $z^{1/p}$.  The output of each step $k$ of the type-$(m,\ell)$ iteration is a rational function $r$ of $A$ with the property that the scalar function $e(z)=(r(z)-z^{1/p})/z^{1/p}$ equioscillates $(m+\ell+1)^k+1$ times on $[\alpha^p,1]$, where $\alpha \in (0,1)$ is a parameter depending on $A$.  With the exception of the Zolotarev iterations (i.e. $p=2$ and $\ell \in \{m-1,m\}$), this equioscillatory behavior does not render $\max_{\alpha^p \le z \le 1} |e(z)|$ minimal among all choices of $r$ with the same numerator and denominator degree.  Nevertheless, we have shown that many of the desirable features of the Zolotarev iterations carry over to the general setting.  A key role in the analysis was played by the asymptotic behavior of rational minimax approximants on short intervals.  

Several topics mentioned in this paper are worth pursuing in more detail.  Remark~\ref{remark:sector} leads naturally to a family of rational minimax iterations for the matrix sector function $\mathrm{sect}_p(A) = A(A^p)^{-1/p}$.  As $\alpha \uparrow 1$, these iterations likely reduce to the Pad\'e iterations for the sector function studied by Laszkiewicz and Zi\k{e}tak~\cite[Section 5]{laszkiewicz2009pade}, so the results therein could inform an analysis of the convergence of the rational minimax iterations on matrices that are non-normal and/or have spectrum away from the positive real axis.  Another topic of interest is computing the action of $A^{1/p}$ on a vector $b$ using rational minimax iterations.  Li and Yang~\cite{li2017interior} address a similar task: computing the action of a spectral filter on $b$ using Zolotarev iterations for $\mathrm{sign}(z)$.  It my may be possible to construct a similar algorithm for computing $A^{1/p}b$.  Finally, the functional iteration~(\ref{pthrootit1}-\ref{pthrootit2}) is of interest in its own right, as it offers a method of rapidly generating rational approximants of $z^{1/p}$ with small relative error, a tool that may have applications in, for instance, numerical conformal mapping~\cite{gopal2018representation}.

\section*{Acknowledgments}

The author was supported in part by the NSF under grant DMS-1703719.

\bibliography{references}

\begin{thebibliography}{10}

\bibitem{akhiezer1956theory}
{\sc N.~I. Akhiezer}, {\em Theory of approximation}, Frederick Ungar Publishing
  Corporation, 1956.

\bibitem{beckermann2013optimally}
{\sc B.~Beckermann}, {\em Optimally scaled {N}ewton iterations for the matrix
  square root}, Advances in Matrix Functions and Matrix Equations workshop,
  Manchester, UK, 2013.

\bibitem{bini2005algorithms}
{\sc D.~A. Bini, N.~J. Higham, and B.~Meini}, {\em Algorithms for the matrix
  pth root}, Numerical Algorithms, 39 (2005), pp.~349--378.

\bibitem{byers2008new}
{\sc R.~Byers and H.~Xu}, {\em A new scaling for {N}ewton's iteration for the
  polar decomposition and its backward stability}, SIAM Journal on Matrix
  Analysis and Applications, 30 (2008), pp.~822--843.

\bibitem{cardoso2011iteration}
{\sc J.~R. Cardoso and A.~F. Loureiro}, {\em Iteration functions for pth roots
  of complex numbers}, Numerical Algorithms, 57 (2011), pp.~329--356.

\bibitem{driscoll2014chebfun}
{\sc T.~A. Driscoll, N.~Hale, and L.~N. Trefethen}, {\em Chebfun guide}, 2014.

\bibitem{gawlik2018zolotarev}
{\sc E.~S. Gawlik}, {\em Zolotarev iterations for the matrix square root},
  arXiv preprint 1804.11000,  (2018).

\bibitem{gawlik2018backward}
{\sc E.~S. Gawlik, Y.~Nakatsukasa, and B.~D. Sutton}, {\em A backward stable
  algorithm for computing the {CS} decomposition via the polar decomposition},
  SIAM Journal on Matrix Analysis and Applications, 39 (2018), pp.~1448--1469.

\bibitem{gomilko2012pade}
{\sc O.~Gomilko, F.~Greco, and K.~Zi\k{e}tak}, {\em A {P}ad{\'e} family of
  iterations for the matrix sign function and related problems}, Numerical
  Linear Algebra with Applications, 19 (2012), pp.~585--605.

\bibitem{gomilko2012regions}
{\sc O.~Gomilko, D.~B. Karp, M.~Lin, and K.~Zi\k{e}tak}, {\em Regions of
  convergence of a {P}ad{\'e} family of iterations for the matrix sector
  function and the matrix pth root}, Journal of Computational and Applied
  Mathematics, 236 (2012), pp.~4410--4420.

\bibitem{gopal2018representation}
{\sc A.~Gopal and L.~N. Trefethen}, {\em Representation of conformal maps by
  rational functions}, arXiv preprint arXiv:1804.08127,  (2018).

\bibitem{guo2010newton}
{\sc C.-H. Guo}, {\em On {N}ewton's method and {H}alley's method for the
  principal pth root of a matrix}, Linear Algebra and its Applications, 432
  (2010), pp.~1905--1922.

\bibitem{guo2006schur}
{\sc C.-H. Guo and N.~J. Higham}, {\em A {S}chur--{N}ewton method for the
  matrix $p$th root and its inverse}, SIAM Journal on Matrix Analysis and
  Applications, 28 (2006), pp.~788--804.

\bibitem{Higham:MCT}
{\sc N.~J. Higham}, {\em The matrix computation toolbox}.
\newblock \url{http://www.ma.man.ac.uk/~higham/mctoolbox}.

\bibitem{higham2008functions}
{\sc N.~J. Higham}, {\em Functions of matrices: {T}heory and computation},
  SIAM, 2008.

\bibitem{higham2011schur}
{\sc N.~J. Higham and L.~Lin}, {\em A {S}chur--{P}ad{\'e} algorithm for
  fractional powers of a matrix}, SIAM Journal on Matrix Analysis and
  Applications, 32 (2011), pp.~1056--1078.

\bibitem{hoskins1979faster}
{\sc W.~Hoskins and D.~Walton}, {\em A faster, more stable method for computing
  the pth roots of positive definite matrices}, Linear Algebra and its
  Applications, 26 (1979), pp.~139--163.

\bibitem{iannazzo2006newton}
{\sc B.~Iannazzo}, {\em On the {N}ewton method for the matrix pth root}, SIAM
  Journal on Matrix Analysis and Applications, 28 (2006), pp.~503--523.

\bibitem{iannazzo2008family}
{\sc B.~Iannazzo}, {\em A family of rational iterations and its application to
  the computation of the matrix pth root}, SIAM Journal on Matrix Analysis and
  Applications, 30 (2008), pp.~1445--1462.

\bibitem{king1971improved}
{\sc R.~F. King}, {\em Improved {N}ewton iteration for integral roots},
  Mathematics of Computation, 25 (1971), pp.~299--304.

\bibitem{laszkiewicz2009pade}
{\sc B.~Laszkiewicz and K.~Zi\k{e}tak}, {\em A {P}ad{\'e} family of iterations
  for the matrix sector function and the matrix pth root}, Numerical Linear
  Algebra with Applications, 16 (2009), pp.~951--970.

\bibitem{li2017interior}
{\sc Y.~Li and H.~Yang}, {\em Interior eigensolver for sparse {H}ermitian
  definite matrices based on {Z}olotarev's functions}, arXiv preprint
  arXiv:1701.08935,  (2017).

\bibitem{maehly1960tschebyscheff}
{\sc H.~Maehly and C.~Witzgall}, {\em Tschebyscheff-approximationen in kleinen
  intervallen {II}}, Numerische Mathematik, 2 (1960), pp.~293--307.

\bibitem{meinardus1980optimal}
{\sc G.~Meinardus and G.~Taylor}, {\em Optimal partitioning of {N}ewton's
  method for calculating roots}, Mathematics of Computation, 35 (1980),
  pp.~1221--1230.

\bibitem{nakatsukasa2010optimizing}
{\sc Y.~Nakatsukasa, Z.~Bai, and F.~Gygi}, {\em Optimizing {H}alley's iteration
  for computing the matrix polar decomposition}, SIAM Journal on Matrix
  Analysis and Applications, 31 (2010), pp.~2700--2720.

\bibitem{nakatsukasa2016computing}
{\sc Y.~Nakatsukasa and R.~W. Freund}, {\em Computing fundamental matrix
  decompositions accurately via the matrix sign function in two iterations:
  {T}he power of {Z}olotarev's functions}, SIAM Review, 58 (2016),
  pp.~461--493.

\bibitem{stahl2003best}
{\sc H.~R. Stahl}, {\em Best uniform rational approximation of $x^\alpha$ on
  [0, 1]}, Acta Mathematica, 190 (2003), pp.~241--306.

\bibitem{trefethen2013approximation}
{\sc L.~N. Trefethen}, {\em Approximation theory and approximation practice},
  vol.~128, SIAM, 2013.

\bibitem{trefethen1983caratheodory}
{\sc L.~N. Trefethen and M.~H. Gutknecht}, {\em The
  {C}arath{\'e}odory-{F}ej{\'e}r method for real rational approximation}, SIAM
  Journal on Numerical Analysis,  (1983), pp.~420--436.

\bibitem{trefethen1985convergence}
{\sc L.~N. Trefethen and M.~H. Gutknecht}, {\em On convergence and degeneracy
  in rational {P}ad{\'e} and {C}hebyshev approximation}, SIAM Journal on
  Mathematical Analysis, 16 (1985), pp.~198--210.

\bibitem{zolotarev1877applications}
{\sc E.~I. Zolotarev}, {\em Applications of elliptic functions to problems of
  functions deviating least and most from zero}, Zapiski St-Petersburg Akad.
  Nauk, 30 (1877), pp.~1--59.

\end{thebibliography}

\end{document}